\documentclass[12pt]{amsart}

\usepackage[matrix,arrow,curve,frame]{xy}
\usepackage{amsmath,amsthm,amssymb,enumerate}
\usepackage{latexsym}
\usepackage{amscd}
\usepackage[colorlinks=false]{hyperref}
\usepackage{euscript}

\newtheorem{thm}[subsection]{Theorem}
\newtheorem{defn}[subsection]{Definition}

\newtheorem{cor}[subsection]{Corollary}
\newtheorem{lemma}[subsection]{Lemma}

\newtheorem{remark}[subsection]{Remark}

\theoremstyle{definition}
\newtheorem{example}[subsection]{Example}

\numberwithin{equation}{section}

\def\quot#1#2{#1/\!\!/#2}

\def\C{\mathbb {C}}

\def\SL{\operatorname{SL}}

\def\phi{{\varphi}}

\def\ra{\rightarrow}
\def\bra{\langle}
\def\ket{\rangle}

\def\cA{{\mathcal A}}
\def\cB{{\mathcal B}}

\def\cD{{\mathcal D}}
\def\cE{{\mathcal E}}

\def\cH{{\mathcal H}}
\def\cI{{\mathcal I}}

\def\cO{{\mathcal O}}

\def\cR{{\mathcal R}}
\def\cS{{\mathcal S}}

\def\cV{{\mathcal V}}
\def\cW{{\mathcal W}}

\def\gg{{\mathfrak g}}

\def\gl{{\mathfrak l}}

\def\go{{\mathfrak o}}
\def\gp{{\mathfrak p}}

\def\gs{{\mathfrak s}}

\DeclareMathOperator{\Sp}{Sp}

\newfont{\german}{eufm10}

\begin{document}
\pagestyle{plain}

\title
{Arc spaces and the vertex algebra commutant problem}

\author{Andrew R. Linshaw, Gerald W. Schwarz, and Bailin Song}
\address{Department of Mathematics, University of Denver}
\email{andrew.linshaw@du.edu}

\address{Department of Mathematics, Brandeis University}
\email{schwarz@brandeis.edu}

\address{Key Laboratory of Wu Wen-Tsun Mathematics, Chinese Academy of Sciences, School of Mathematical Sciences, University of Science and Technology of China}
\email{bailinso@ustc.edu.cn}

\begin{abstract} Given a vertex algebra $\cV$ and a subalgebra $\cA\subset \cV$, the commutant $\text{Com}(\cA,\cV)$ is the subalgebra of $\cV$ which commutes with all elements of $\cA$. This construction is analogous to the ordinary commutant in the theory of associative algebras, and is important in physics in the construction of coset conformal field theories. When $\cA$ is an affine vertex algebra, $\text{Com}(\cA,\cV)$ is closely related to rings of invariant functions on arc spaces. We find strong finite generating sets for a family of examples where $\cA$ is affine and $\cV$ is a $\beta\gamma$-system, $bc$-system, or $bc\beta\gamma$-system.

\end{abstract}

\maketitle
\section{Introduction}

Let $\cV$ be a vertex algebra, and let $\cA$ be a subalgebra of $\cV$. The {\it commutant} of $\cA$ in $\cV$, denoted by $\text{Com}(\cA,\cV)$, is the subalgebra consisting of all elements $v\in\cV$ such that $ [a(z),v(w)] = 0$ for all $a\in \cA$. This construction was introduced by Frenkel and Zhu in \cite{FZ}, generalizing earlier constructions in representation theory \cite{KP} and physics \cite{GKO}, and is important in the construction of coset conformal field theories. It is also natural to study the double commutant $\text{Com}(\text{Com}(\cA,\cV),\cV)$, which always contains $\cA$. If $\cA = \text{Com}(\text{Com}(\cA,\cV),\cV)$, we say that $\cA$ and $\text{Com}(\cA,\cV)$ form a {\it Howe pair} inside $\cV$. If $\cA$ acts semisimply on $\cV$, $\text{Com}(\cA,\cV)$ can be studied by decomposing $\cV$ as an $\cA$-module. Otherwise, there are few existing techniques for studying commutants, and there are very few examples where an exhaustive description can be given in terms of generators, operator product expansions, and normally ordered polynomial relations among the generators. 

An equivalent definition of $\text{Com}(\cA,\cV)$ is the set of elements $v\in \cV$ such that $a\circ_n v = 0$ for all $a\in\cA$ and $n\geq 0$. We may regard $\text{Com}(\cA,\cV)$ as the algebra of  invariants in $\cV$ under the action of $\cA$. If $\cA$ is a homomorphic image of an affine vertex algebra associated to some Lie algebra $\gg$, $\text{Com}(\cA,\cV)$ is just the invariant space $\cV^{\gg[t]}$, and in this case one can apply techniques from invariant theory and commutative algebra. This approach was introduced in \cite{LL}, and the structure that makes it work is a {\it good increasing filtration} on $\cV$. The associated graded object $\text{gr}(\cV)$ is then an abelian vertex algebra, i.e., a (super)commutative ring with a differential, and in many cases it can be interpreted as the ring $\cO(X_{\infty})$ of functions on the arc space $X_{\infty}$ of some scheme $X$. For example, the level $k$ universal affine vertex algebra $V_k(\gg)$ of a Lie algebra $\gg$ has a filtration for which $\text{gr}(V_k(\gg)) \cong \cO(\gg_{\infty})$ and the $\beta\gamma$-system $\cS(V)$ of a finite-dimensional vector space $V$ has $\text{gr}(\cS(V)) \cong \cO((V\oplus V^*)_{\infty})$. For any $X$, $\cO(X_{\infty})$ is an abelian vertex algebra \cite{FBZ}, and one of the themes of this paper is that the geometry of arc spaces can be used to answer structural questions about nonabelian vertex algebras as well.

If $\cA\subset \cV$ is an affine vertex algebra and $\cV$ has a $\gg[t]$-invariant good increasing filtration, there is an action of $\gg[t]$ on $\text{gr}(\cV)$ by derivations of degree zero. There is an injective map of differential (super)commutative algebras 
\begin{equation} \label{first} \text{gr}(\cV^{\gg[t]})\hookrightarrow \text{gr}(\cV)^{\gg[t]}.\end{equation} Unfortunately, the associated graded functor and the invariant functor need not commute with each other, and this map is generally not an isomorphism. The structure of $\text{gr}(\cV)^{\gg[t]}$ is simpler than that of $\text{gr}(\cV^{\gg[t]})$, and is closely related to rings of invariant functions on arc spaces. If $\text{gr}(\cV)^{\gg[t]}$ is finitely generated as a differential algebra (which is the case in our main examples), checking the surjectivity of (\ref{first}) becomes a finite problem. If (\ref{first}) is surjective, there is a reconstruction theorem that says that a generating set for $\text{gr}(\cV)^{\gg[t]}$ as a differential algebra corresponds to a strong generating set for $\cV^{\gg[t]}$ as a vertex algebra. Moreover, all normally ordered polynomial relations among the generators of $\cV^{\gg[t]}$ correspond to classical relations in $\text{gr}(\cV)^{\gg[t]}$, with suitable quantum corrections.

In our main examples, $\cV$ is the $\beta\gamma$-system $\cS(V)$, which is the vertex algebra analogue of the Weyl algebra $\cD(V)$. If $G$ is a connected, reductive Lie group with Lie algebra $\gg$ acting on $V$, there is an action of the corresponding affine vertex algebra on $\cS(V)$, and the commutant $\cS(V)^{\gg[t]}$ is analogous to the invariant ring $\cD(V)^G$. The injective map 
\begin{equation} \label{firstsv} \text{gr}(\cS(V)^{\gg[t]}) \hookrightarrow \text{gr}(\cS(V))^{\gg[t]} \cong \cO((V\oplus V^*)_{\infty})^{\gg[t]} = \cO((V\oplus V^*)_{\infty})^{G_{\infty}}\end{equation} is generally not an isomorphism, and $\cO((V\oplus V^*)_{\infty})^{G_{\infty}}$ generally has a complicated structure and need not be finitely generated as a differential algebra. In our main examples, $G = \text{SL}_n$, $\text{GL}_n$, $\text{SO}_n$, or $\text{Sp}_{2n}$ and $V$ is a sum of copies of the standard representation such that $(V\oplus V^*)/\!\!/G$ is either smooth or a complete intersection. In these cases, $\cO((V\oplus V^*)_{\infty})^{G_{\infty}} \cong \cO(((\quot{V\oplus V^*)}G)_{\infty})$, which is generated by $\cO(V\oplus V^*)^G$ as a differential algebra. Moreover, (\ref{firstsv}) is an isomorphism, so the generators of $\cO(V\oplus V^*)^G$ correspond to strong generators of $\cS(V)^{\gg[t]}$ as a vertex algebra. We also consider the double commutant $\text{Com}(\cS(V)^{\gg[t]},\cS(V))$ and we find some analogues of classical $\text{GL}_n - \text{GL}_m$, $\text{SO}_n-\gs\gp_{2m}$, and $\text{Sp}_{2n} - \gs\go_{2m}$ Howe duality in this setting \cite{H}. Finally, we use similar techniques to describe some commutant subalgebras of $bc$-systems and $bc\beta\gamma$-systems.

Our methods are insufficient to describe $\cS(V)^{\gg[t]}$ in the general case without further refinement. However, there is evidence that $\cS(V)^{\gg[t]}$ is better behaved than its classical counterpart $\text{gr}(\cS(V))^{\gg[t]}$. Suppose that $G$ is a torus acting faithfully on $V$. In this case, \eqref{first} is not surjective, but $\cS(V)^{\gg[t]}$ is always strongly finitely generated as a vertex algebra (see Theorem 12 of \cite{LII}). By contrast, $\text{gr}(\cS(V))^{\gg[t]}$ need not be finitely generated as a differential algebra; Corollary 3.14 of \cite{LSS} shows that this fails whenever $V$ possesses a nontrivial slice representation of a finite group. Similarly, in the case where $G$ is finite, so that $G = G_{\infty}$, $\cS(V)^G$ is always strongly finitely generated (see Theorem 10 of \cite{LII}), whereas $\text{gr}(\cS(V))^G$ is not finitely generated as a differential algebra unless $G$ acts trivially on $V$ (see Theorem 3.13 of \cite{LSS}). Based on these examples, we expect that $\cS(V)^{\gg[t]}$ will be strongly finitely generated under fairly general circumstances, but in most cases the generating set will not correspond naively to a generating set for $\cO(V\oplus V^*)^G$.

\section{Vertex algebras}
We begin with a short introduction to vertex algebras \cite{B,FLM}, following the formalism developed in \cite{LZ} and partly in \cite{LiI}. Let $V=V_0\oplus V_1$ be a super vector space over $\mathbb{C}$, and let $z,w$ be formal variables. Let $\text{QO}(V)$ denote the space of linear maps $$V\ra V((z)):=\{\sum_{n\in\mathbb{Z}} v(n) z^{-n-1}|
v(n)\in V,\ v(n)=0\ \text{for} \  n >\!\!> 0 \}.$$ Each $a\in \text{QO}(V)$ can be
uniquely represented as a power series
$$a=a(z):=\sum_{n\in\mathbb{Z}}a(n)z^{-n-1}\in \text{End} (V)[[z,z^{-1}]].$$ We
call $a(n)$ the $n^{\text{th}}$ Fourier mode of $a(z)$. Each $a\in
\text{QO}(V)$ is assumed to be of the form $a=a_0+a_1$ where $a_i:V_j\ra V_{i+j}((z))$ for $i,j\in\mathbb{Z}/2\mathbb{Z}$, and we write $|a_i| = i$.

There is a set of nonassociative bilinear operations
$\circ_n$ on $\text{QO}(V)$, indexed by $n\in\mathbb{Z}$, which we call the $n^{\text{th}}$ circle
products. For homogeneous $a,b\in \text{QO}(V)$, they are defined by
$$
a(w)\circ_n b(w)=\text{Res}_z a(z)b(w)~\iota_{|z|>|w|}(z-w)^n-
(-1)^{|a||b|}\text{Res}_z b(w)a(z)~\iota_{|w|>|z|}(z-w)^n.
$$
Here $\iota_{|z|>|w|}f(z,w)\in\mathbb{C}[[z,z^{-1},w,w^{-1}]]$ denotes the
power series expansion of a rational function $f$ in the region
$|z|>|w|$. We usually omit the symbol $\iota_{|z|>|w|}$ and just
write $(z-w)^{-1}$ to mean the expansion in the region $|z|>|w|$,
and write $-(w-z)^{-1}$ to mean the expansion in $|w|>|z|$. It is
easy to check that $a(w)\circ_n b(w)$ above is a well-defined
element of $\text{QO}(V)$.

For $a,b\in \text{QO}(V)$, the {\it operator product expansion} (OPE) formula is the following identity of formal power series.
 \begin{equation}\label{opeform} a(z)b(w)=\sum_{n\geq 0}a(w)\circ_n
b(w)~(z-w)^{-n-1}+:a(z)b(w):. \end{equation} 
Here $:a(z)b(w):\ =a(z)_-b(w)\ +\ (-1)^{|a||b|} b(w)a(z)_+$, where $a(z)_-=\sum_{n<0}a(n)z^{-n-1}$ and $a(z)_+=\sum_{n\geq
0}a(n)z^{-n-1}$. Equation \eqref{opeform} is often written in the form
$$a(z)b(w)\sim\sum_{n\geq 0}a(w)\circ_n b(w)\ (z-w)^{-n-1},$$ where
$\sim$ means equal modulo the term $:a(z)b(w):$, which is regular at $z=w$. 

Note that $:a(w)b(w):$ is a well-defined element of
$\text{QO}(V)$. It is called the Wick product (or normally ordered product) of $a$ and $b$, and it
coincides with $a(w)\circ_{-1}b(w)$. The other negative circle products are related to this by
$$ n!\ a(z)\circ_{-n-1}b(z)=\ :(\partial^n a(z))b(z):,$$ where $\partial$ denotes the formal differentiation operator $\frac{d}{dz}$. The iterated Wick product of $a_1(z),\dots ,a_k(z)\in \text{QO}(V)$ is defined inductively by
\begin{equation}\label{iteratedwick} :a_1(z)a_2(z)\cdots a_k(z): \ =\  :a_1(z)b(z):,\end{equation} where $b(z)=\ :a_2(z)\cdots a_k(z):$. We often omit the formal variable $z$ when no confusion can arise.

The set $\text{QO}(V)$ is a nonassociative algebra with the operations
$\circ_n$ and a unit $1$. We have $1\circ_n a=\delta_{n,-1}a$ for
all $n$, and $a\circ_n 1=\delta_{n,-1}a$ for $n\geq -1$. A subspace $\cA\subset \text{QO}(V)$ containing $1$ which is closed under the circle products will be called a quantum operator algebra (QOA).
In particular, $\cA$ is closed under $\partial$
since $\partial a=a\circ_{-2}1$. Many formal algebraic
notions are immediately clear: a homomorphism is just a linear
map that sends $1$ to $1$ and preserves all circle products; a module over $\cA$ is a
vector space $M$ equipped with a homomorphism $\cA\rightarrow
\text{QO}(M)$, etc. A subset $S=\{a_i|\ i\in I\}$ of $\cA$ is said to generate $\cA$ if every element $a\in\cA$ can be written as a linear
combination of nonassociative words in the letters $a_i$, $\circ_n$, for
$i\in I$ and $n\in\mathbb{Z}$. We say that $S$ {\it strongly generates} $\cA$ if any $a\in\cA$ can be written as a linear combination of words in the letters $a_i$, $\circ_n$ for $n<0$. Equivalently, $\cA$ is spanned by the set of normally ordered monomials $\{ :\partial^{k_1} a_{i_1}(z)\cdots \partial^{k_m} a_{i_m}(z):| \ i_1,\dots,i_m \in I,~ k_1,\dots,k_m \geq 0\}$.

We say that $a,b\in \text{QO}(V)$ {\it quantum commute} if $(z-w)^N
[a(z),b(w)]=0$ for some $N\geq 0$. Here $[,]$ denotes the super bracket. This condition implies that $a\circ_n b = 0$ for $n\geq N$, so (\ref{opeform}) becomes a finite sum. A {\it commutative quantum operator algebra} (CQOA) is a QOA whose elements pairwise quantum commute. Finally, the notion of a CQOA is equivalent to the notion of a vertex algebra. Every CQOA $\cA$ is itself a faithful $\cA$-module, called the {\it left regular
module}. Define
$$\rho:\cA\rightarrow \text{QO}(\cA),\qquad a\mapsto\hat a, \qquad \hat
a(\zeta)b=\sum_{n\in\mathbb{Z}} (a\circ_n b)~\zeta^{-n-1}.$$ Then $\rho$ is an injective QOA homomorphism,
and the quadruple $(\cA,\rho,1,\partial)$ is a vertex
algebra in the sense of \cite{FLM}. Conversely, if $(V,Y,{\bf 1},D)$ is
a vertex algebra, $Y(V)\subset \text{QO}(V)$ is a
CQOA. {\it We will refer to a CQOA simply as a
vertex algebra throughout this paper}.

\begin{example}[Affine vertex algebras] Let $\gg$ be a finite-dimensional, complex Lie algebra, equipped with a symmetric, invariant bilinear form $B$. The loop algebra $\gg[t,t^{-1}] = \gg\otimes \mathbb{C}[t,t^{-1}]$ has a one-dimensional central extension $\hat{\gg} = \gg[t,t^{-1}]\oplus \mathbb{C}\kappa$ determined by $B$, with bracket $$[\xi t^n, \eta t^m] = [\xi,\eta] t^{n+m} + n B(\xi,\eta) \delta_{n+m,0} \kappa,$$ and $\mathbb{Z}$-gradation $\text{deg}(\xi t^n) = n$, $\text{deg}(\kappa) = 0$. Let $\hat{\gg}_{\geq 0} = \bigoplus_{n\geq 0} \hat{\gg}_n$ where $\hat{\gg}_n$ denotes the subspace of degree $n$, and let $C$ be the one-dimensional $\hat{\gg}_{\geq 0}$-module on which $\xi t^n$ acts trivially for $n\geq 0$, and $\kappa$ acts by $k$ times the identity. Define $V = U(\hat{\gg})\otimes_{U(\hat{\gg}_{\geq 0})} C$, and let $X^{\xi}(n)\in \text{End}(V)$ be the linear operator representing $\xi t^n$ on $V$. Define $X^{\xi} (z) = \sum_{n\in\mathbb{Z}} X^{\xi} (n) z^{-n-1}$, which lies in $\text{QO}(V)$ and satisfies $$X^{\xi}(z)X^{\eta} (w)\sim kB(\xi,\eta) (z-w)^{-2} + X^{[\xi,\eta]}(w) (z-w)^{-1} .$$ The vertex algebra $V_k(\gg,B)$ generated by $\{X^{\xi}| \ \xi \in\gg\}$ is known as the {\it universal affine vertex algebra} associated to $\gg$ and $B$ at level $k$. If $\gg$ is a simple Lie algebra, we will always take $B$ to be the normalized Killing form $\frac{1}{2 h^{\vee}} \bra, \ket_K$, where $h^{\vee}$ is the dual Coxeter number, and we use the notation $V_k(\gg)$ instead of $V_k(\gg,B)$, as is customary. 

For $k \neq - h^{\vee}$, $V_k(\gg)$ possesses the Sugawara conformal vector
\begin{equation}
L = \frac{1}{2(k+h^\vee)}\sum_{i=1}^n :X^{\xi_i}X^{\xi_i}:
\end{equation}
of central charge $c= \frac{k\text{dim}(\gg)}{k+h^\vee}$. Here $\{\xi_1,\dots, \xi_n\}$ is an orthonormal basis for $\gg$. With respect to $L$, $V_k(\gg)$ has a weight grading by $\mathbb{Z}_{\geq 0}$ such that each $X^{\xi_i}$ is primary of weight one. At the critical level $k = -h^{\vee}$, the Lie subalgebra spanned by $\{L_n| n\geq -1\}$ still acts on $V_{-h^{\vee}}(\gg)$, and the same weight grading exists.

Suppose that $\gg = \C$, regarded as an abelian Lie algebra, and that $B$ is nondegenerate. Then for $k\neq 0$, $V_k(\gg,B)$ is just the Heisenberg vertex algebra $\cH$ with generator $\alpha$ satisfying $\alpha(z) \alpha(w) \sim (z-w)^{-2}$. There is an analogous Virasoro element $L = \frac{1}{2} :\alpha \alpha:$, under which $\alpha$ is primary of weight one. \end{example}

\begin{example}[$\beta\gamma$-systems] Let $V$ be a finite-dimensional complex vector space. The $\beta\gamma$-system or algebra of chiral differential operators $\cS(V)$ was introduced in \cite{FMS}. It has even generators $\beta^{x}$, $\gamma^{x'}$ for $x\in V$, $x'\in V^*$, which satisfy
\begin{equation}\beta^x(z)\gamma^{x'}(w)\sim\langle x',x\rangle (z-w)^{-1}, \qquad \gamma^{x'}(z)\beta^x(w)\sim -\langle x',x\rangle (z-w)^{-1},\end{equation}
$$\beta^x(z)\beta^y(w)\sim 0, \qquad \gamma^{x'}(z)\gamma^{y'}(w)\sim 0.$$ Here $\bra,\ket$ denotes the natural pairing between $V^*$ and $V$. We give $\cS(V)$ the conformal structure \begin{equation} \label{virasorobg} L_{\cS}= \sum_{i=1}^n :\beta^{x_i}\partial\gamma^{x'_i}:,\end{equation} under which $\beta^{x_i}$, $\gamma^{x'_i}$ are primary of weights $1$, $0$, respectively. Here $\{x_1,\dots,x_n\}$ is a basis for $V$ and $\{x'_1,\dots,x'_n\}$ is the dual basis for $V^*$. There is a basis for $\cS(V)$ consisting of iterated Wick products of the generators and their derivatives. There is an additional $\mathbb{Z}$-grading on $\cS(V)$ which we call the $\beta\gamma$-charge. Define \begin{equation}\label{defcharge} e = \sum_{i=1}^n : \beta^{x_i}\gamma^{x'_i}:.\end{equation} The zero mode $e(0)$ acts diagonalizably on $\cS(V)$. The $\beta\gamma$-charge grading is just the eigenspace decomposition of $\cS(V)$ under $e(0)$, and $\beta^{x}$, $\gamma^{x'}$ have $\beta\gamma$-charges $-1$, $1$, respectively. \end{example}

Given a vertex algebra $\cV$ and a subalgebra $\cA$, the commutant $\text{Com}(\cA,\cV)$ was introduced by Frenkel and Zhu in \cite{FZ}, generalizing a previous construction known as the coset construction. 
\begin{defn} Let $\cV$ be a vertex algebra, and let $\cA$ be a subalgebra. The commutant of $\cA$ in $\cV$, denoted by $\text{Com}(\cA,\cV)$, is the subalgebra of vertex operators $v\in\cV$ such that $[a(z),v(w)] = 0$ for all $a\in\cA$. Equivalently, $a\circ_n v = 0$ for all $a\in\cA$ and $n\geq 0$.\end{defn}
We regard $\cV$ as a module over $\cA$ via the left regular action, and we regard $\text{Com}(\cA,\cV)$ as the invariant subalgebra. If $\cA$ is a homomorphic image of $V_k(\gg,B)$, $\text{Com}(\cA,\cV) = \cV^{\gg[t]}$. The double commutant $\text{Com}(\text{Com}(\cA,\cV),\cV)$ always contains $\cA$, and if $\cA = \text{Com}(\text{Com}(\cA,\cV),\cV)$, we say that $\cA$ and $\text{Com}(\cA,\cV)$ form a {\it Howe pair} inside $\cV$. Since $$\text{Com}(\text{Com}(\text{Com}(\cA,\cV),\cV),\cV) = \text{Com}(\cA,\cV),$$ a subalgebra $\cB$ is a member of a Howe pair if and only if $\cB = \text{Com}(\cA,\cV)$ for some $\cA$.

\subsection{The Zhu functor}
Let $\cV$ be a vertex algebra with weight grading $\cV = \bigoplus_{n\in\mathbb{Z}} \cV_n$. The Zhu functor \cite{Zh} attaches to $\cV$ an associative algebra $A(\cV)$, together with a surjective linear map $\pi_{\text{Zhu}}:\cV\ra A(\cV)$. For $a\in \cV_m$, and $b\in\cV$, define $$a*b = \text{Res}_z \bigg (a(z) \frac{(z+1)^{m}}{z}b\bigg),$$ and extend $*$ by linearity to a bilinear operation $\cV\otimes \cV\ra \cV$. Let $O(\cV)$ denote the subspace of $\cV$ spanned by elements of the form $$a\circ b = \text{Res}_z \bigg (a(z) \frac{(z+1)^{m}}{z^2}b\bigg),$$ for $a\in\cV_m$, and let $A(\cV)$ be the quotient $\cV/O(\cV)$, with projection $\pi_{\text{Zhu}}:\cV\ra A(\cV)$. For $a,b\in \cV$, $a\sim b$ means $a-b\in O(\cV)$, and $[a]$ denotes the image of $a$ in $A(\cV)$. 
\begin{thm} (Zhu) $O(\cV)$ is a two-sided ideal in $\cV$ under the product $*$, and $(A(\cV),*)$ is an associative algebra with unit $[1]$. The assignment $\cV\mapsto A(\cV)$ is functorial.\end{thm}
Let $\mathcal{V}$ be a vertex algebra which is strongly generated by a set of weight-homogeneous elements $\alpha_i$ of weights $w_i$, for $i$ in some index set $I$. Then $A(\mathcal{V})$ is generated by $\{ a_i = \pi_{\text{Zhu}}(\alpha_i(z))| i\in I\}$. The main application of the Zhu functor is to study the representation theory of $\cV$. A $\mathbb{Z}_{\geq 0}$-graded module $M = \bigoplus_{n\geq 0} M_n$ over $\cV$ is called {\it admissible} if for every $a\in\cV_m$, $a(n) M_k \subset M_{m+k -n-1}$, for all $n\in\mathbb{Z}$. Given $a\in\cV_m$, the Fourier mode $a(m-1)$ acts on each $M_k$. The subspace $M_0$ is then a module over $A(\cV)$ with action $[a]\mapsto a(m-1) \in \text{End}(M_0)$. In fact, $M\mapsto M_0$ provides a bijection between irreducible, admissible $\cV$-modules and irreducible $A(\cV)$-modules.

The Zhu functor and the commutant construction interact in the following way: for any subalgebra $\cA\subset \cV$, we have a commutative diagram 
\begin{equation}\label{cdgencase} \begin{array}[c]{ccc}
\text{Com}(\cA,\cV)&\stackrel{}{\hookrightarrow}& \cV  \\
\downarrow\scriptstyle{\pi}&&\downarrow\scriptstyle{\pi_{\text{Zhu}}}\\
\text{Com}(A,A(\cV))&\stackrel{}{\hookrightarrow}& A(\cV)
\end{array} .\end{equation} 
Here $A=\pi_{\text{Zhu}}(\cA)\subset A(\cV)$, and $\text{Com}(A,A(\cV))$ is the ordinary commutant in the theory of associative algebras. The horizontal maps are inclusions, and $\pi$ is the restriction of $\pi_{\text{Zhu}}$ to $\text{Com}(\cA,\cV)$. In general, the map $\pi$ need not be surjective and $A(\text{Com}(\cA,\cV))$ need not coincide with $\text{Com}(A,A(\cV))$. However, both these statements are true in the main examples in this paper.

\subsection{Invariant chiral differential operators}
The main examples of commutants that we consider are analogous to classical rings of invariant differential operators. Let $V=\mathbb{C}^n$, and fix a basis $\{x_1,\dots,x_n\}$ for $V$, with dual basis $\{x'_1,\dots,x'_n\}$ for $V^*$. The {\it Weyl algebra} $\cD(V)$ is generated by $x_i', \frac{\partial}{\partial x'_i}$, which satisfy $[\frac{\partial}{\partial x'_i},x'_j] = \delta_{i,j}$. Equip $\cD(V)$ with the Bernstein filtration \begin{equation}\label{bernstein} \cD(V)_{(0)}\subset \cD(V)_{(1)}\subset  \cdots,\end{equation} defined by $(x'_1)^{k_1} \cdots (x'_n)^{k_n} (\frac{\partial}{\partial x'_1})^{l_1}\cdots (\frac{\partial}{\partial x'_n})^{l_n} \in \cD(V)_{(r)}$ if $k_1 + \cdots +k_n + l_1 + \cdots +l_n \leq r$. Given $\omega\in \cD(V)_{(r)}$ and $\nu\in\cD(V)_{(s)}$, $[\omega,\nu]\in \cD(V)_{(r+s-2)}$, so that \begin{equation}\label{isodi} \text{gr}(\cD(V)) = \bigoplus_{r\geq 0} \cD(V)_{(r)} / \cD(V)_{(r-1)} \cong \text{Sym}(V\oplus V^*),\end{equation} where $\cD(V)_{(-1)} = \{0\}$. We say that $\text{deg}(\alpha) = d$ if $\alpha\in\cD(V)_{(d)}$ and $\alpha\notin \cD(V)_{(d-1)}$. 

Let $G$ be a connected, reductive complex algebraic group with Lie algebra $\gg$, and let $V$ be a $G$-module. Then $G$ acts on $\cD(V)$ and preserves the filtration. The induced action $\rho: \gg \ra \text{Der}(\cD(V))$ can be realized by {\it inner} derivations: there is a Lie algebra homomorphism \begin{equation}\label{defoftau} \tau:\gg\ra \cD(V), \qquad \xi\mapsto - \sum_{i=1}^n x'_i \rho(\xi)\big(\frac{\partial}{\partial x'_i}\big).\end{equation} Here $\rho(\xi)$ acts on the span of the operators $\frac{\partial}{\partial x'_i}$, regarded as a copy of $V$. Given $\xi\in\gg$, $\tau(\xi)$ is just the vector field on $V$ generated by $\xi$, and $\xi$ acts on $\cD(V)$ by $[\tau(\xi),-]$. We can extend $\tau$ to a map $U(\gg)\ra \cD(V)$. Since $G$ is connected, $\cD(V)^G=\text{Com}(\tau(U(\gg)),\cD(V))$. Finally, (\ref{bernstein}) restricts to a filtration
$\cD(V)^{G}_{(0)}\subset \cD(V)^{G}_{(1)} \subset \cdots$ on $\cD(V)^{G}$, and $\text{gr}(\cD(V)^{G}) \cong \text{gr}(\cD(V))^{G} \cong \text{Sym}(V\oplus V^*)^{G}$. 

It is well known that $\cD(V)$ is the Zhu algebra of the $\beta\gamma$-system $\cS(V)$, and that $U(\gg)$ is the Zhu algebra of the affine vertex algebra $V_k(\gg,B)$. The map $\rho$ induces a vertex algebra homomorphism $\hat{\tau}: V_{-1}(\gg,B)\ra \cS(V)$, which is analogous to (\ref{defoftau}), given by
\begin{equation} \label{deftheta} \hat{\tau}(X^{\xi}) = \theta^{\xi}_{\cS} = - \sum_{i=1}^n :\gamma^{x'_i}\beta^{\rho(\xi)(x_i)}:,\end{equation} where $B$ is the bilinear form $B(\xi,\eta) = \text{Tr}(\rho(\xi)\rho(\eta))$. We have a commutative diagram 
\begin{equation}\label{commdiagfirst} \begin{array}[c]{ccc}
V_{-1}(\gg,B) &\stackrel{}{\rightarrow}& \cS(V)  \\
\downarrow\scriptstyle{\pi_{\text{Zhu}} }&&\downarrow\scriptstyle{\pi_{\text{Zhu}}}\\
U(\gg) &\stackrel{}{\rightarrow}& \cD(V)
\end{array}.\end{equation} The top horizontal map is $\hat{\tau}$, and the bottom map coincides with $\tau$ up to modification by a scalar (i.e., an element of degree zero in $\cD(V)$). Let $\Theta$ denote the subalgebra $\hat{\tau}(V_{-1}(\gg,B))\subset \cS(V)$. The commutant $\text{Com}(\Theta,\cS(V))= \cS(V)^{\gg[t]}$ will be called the algebra of {\it invariant chiral differential operators} on $V$. By (\ref{cdgencase}), there is another commutative diagram
\begin{equation}\label{commdiag} \begin{array}[c]{ccc}
\cS(V)^{\gg[t]}&\stackrel{}{\rightarrow}& \cS(V)  \\
\downarrow\scriptstyle{\pi}&&\downarrow\scriptstyle{\pi_{\text{Zhu}}}\\
\cD(V)^{G} &\stackrel{}{\rightarrow}& \cD(V)
\end{array}.\end{equation}

\section{Graded and filtered structures}
Let $\cR$ be the category of vertex algebras $\cA$ equipped with a $\mathbb{Z}_{\geq 0}$-filtration
\begin{equation} \cA_{(0)}\subset\cA_{(1)}\subset\cA_{(2)}\subset \cdots, \qquad \cA = \bigcup_{k\geq 0}
\cA_{(k)}\end{equation} such that $\cA_{(0)} = \mathbb{C}$, and for all
$a\in \cA_{(k)}$, $b\in\cA_{(l)}$, we have
\begin{equation} \label{goodi} a\circ_n b\in\cA_{(k+l)}, \qquad \text{for}\
n<0,\end{equation}
\begin{equation} \label{goodii} a\circ_n b\in\cA_{(k+l-1)}, \qquad \text{for}\
n\geq 0.\end{equation}
Elements $a(z)\in\cA_{(d)}\setminus \cA_{(d-1)}$ are said to have degree $d$.

Filtrations on vertex algebras satisfying (\ref{goodi})-(\ref{goodii}) were introduced in \cite{LiII}, and are known as {\it good increasing filtrations}. Setting $\cA_{(-1)} = \{0\}$, the associated graded object $\text{gr}(\cA) = \bigoplus_{k\geq 0}\cA_{(k)}/\cA_{(k-1)}$ is a
$\mathbb{Z}_{\geq 0}$-graded associative, supercommutative algebra with a
unit $1$ under a product induced by the Wick product on $\cA$. For each $r\geq 1$ we have the projection \begin{equation} \label{projmap}\phi_r: \cA_{(r)} \ra \cA_{(r)}/\cA_{(r-1)}\subset \text{gr}(\cA).\end{equation} 
Moreover, $\text{gr}(\cA)$ has a derivation $\partial$ of degree zero
(induced by the operator $\partial = \frac{d}{dz}$ on $\cA$), and
for each $a\in\cA_{(d)}$ and $n\geq 0$, the operator $a\circ_n$ on $\cA$
induces a derivation of degree $d-k$ on $\text{gr}(\cA)$, which we denote by $a(n)$. Here $$k  = \text{sup} \{ j\geq 1|~ \cA_{(r)}\circ_n \cA_{(s)}\subset \cA_{(r+s-j)}~\forall r,s,n\geq 0\},$$ as in \cite{LL}. Finally, these derivations give $\text{gr}(\cA)$ the structure of a vertex Poisson algebra.

The assignment $\cA\mapsto \text{gr}(\cA)$ is a functor from $\cR$ to the category of $\mathbb{Z}_{\geq 0}$-graded supercommutative rings with a differential $\partial$ of degree 0, which we will call $\partial$-rings. A $\partial$-ring is just an abelian vertex algebra, that is, a vertex algebra $\cV$ in which $[a(z),b(w)] = 0$ for all $a,b\in\cV$. A $\partial$-ring $A$ is said to be generated by a subset $\{a_i| i\in I\}$ if $\{\partial^k a_i| i\in I, k\geq 0\}$ generates $A$ as a graded ring. The key feature of $\cR$ is the following reconstruction property \cite{LL}.

\begin{lemma}\label{reconlem}Let $\cA$ be a vertex algebra in $\cR$ and let $\{a_i| i\in I\}$ be a set of generators for $\text{gr}(\cA)$ as a $\partial$-ring, where $a_i$ is homogeneous of degree $d_i$. If $a_i(z)\in\cA_{(d_i)}$ are vertex operators such that $\phi_{d_i}(a_i(z)) = a_i$, then $\cA$ is strongly generated as a vertex algebra by $\{a_i(z)| i\in I\}$.\end{lemma}

The filtration can also be used to study normally ordered polynomial relations among the generators of a vertex algebra in $\cR$. Let $\cA$, $\{a_i| i\in I\}$, and $a_i(z)\in\cA_{(d_i)}$ be as above. Given a homogeneous polynomial $p\in \text{gr}(\cA)$ of degree $d$, a {\it normal ordering} of $p$ will be a choice of normally ordered polynomial $p(z)\in \cA_{(d)}$, obtained by replacing each $a_i$ by $a_i(z)$, and replacing ordinary products with iterated Wick products. Of course $p(z)$ is not unique, but for any choice of $p(z)$ we have $\phi_{d}(p(z)) = p$.

Suppose that $p$ is a relation among the $a_i$ and their derivatives, which we assume to be homogeneous of degree $d$. Let $p^d(z) \in \cA$ be some normal ordering of $p$. Since $p$ is identically zero, $p^d(z)\in\cA_{(d-1)}$. The polynomial $\phi_{d-1} (p^d(z)) \in \text{gr}(\cA)$ is homogeneous of degree $d-1$; if it is nonzero, it can be expressed as a polynomial in the $a_i$'s and their derivatives. Choose some normal ordering of this polynomial, and call it $-p^{d-1}(z)$. Then $p^{d} (z)+ p^{d-1}(z)$ has the property that $$ \phi_d (p^d (z)+ p^{d-1}(z)) = p, \qquad  p^d(z) + p^{d-1}(z) \in \cA_{(d-2)}.$$ Continuing this process, we arrive at a vertex operator $p(z)=\sum_{k=1}^{d} p^{k}(z) \in \cA$, which is identically zero. We view $p(z)$ as a quantum correction of the relation $p$.

Let $\cI$ denote the ideal of relations among $\{\partial^k a_i| i\in I, k\geq 0\}$, which is clearly a $\partial$-ideal, i.e., it is closed under $\partial$. A set $\{r_j| j\in J\}$ such that $r_j$ is homogeneous of degree $e_j$ is said to generate $\cI$ as a $\partial$-ideal if $\{\partial^k r_j| j\in J,\ k\geq 0\}$ generates $\cI$ as an ideal. Let $r_j(z)$ be the relation in $\cA$ of degree $e_j$ which is a quantum correction of $r_j$, as above.

\begin{lemma} \label{idealrecon} All normally ordered polynomial relations among the $a_i(z)$'s and their derivatives are consequences of $\{r_j(z)|j\in J\}$ in the sense that they can be written as linear combinations of vertex operators of the form $:\alpha(z) \partial^k r_j(z):$ for $\alpha \in \cA$ and $k\geq 0$.\end{lemma}

\begin{proof} 
Let $\omega(z)$ be a normally ordered relation of degree $d$, meaning that $d$ is the maximal degree of monomials appearing in $\omega(z)$. Then $\phi_d(\omega(z)) = 0$, and can be expressed in the form $$\sum_{j\in J}\sum_{k \geq 0} \alpha_{jk} \partial^k r_j,$$ where all but finitely many $\alpha_{jk}$ are zero, and each $\alpha_{jk}$ is homogeneous of degree $d-e_j$. Choose vertex operators $\alpha_{jk}(z)\in \cA_{(d-e_j)}$ such that $\phi_{d-e_j}(\alpha_{jk}(z)) = \alpha_{jk}$, and let $$\omega'(z) = \sum_{j\in J}\sum_{k \geq 0} :\alpha_{jk}(z) \partial^k r_j(z):.$$ Since each $r_j(z)$ is identically zero, $\omega''(z) = \omega(z) - \omega'(z)$ is also identically zero. Since $\omega'(z)$ has the desired form and $\omega''(z)$ is a relation of degree at most $d-1$, the claim follows by induction on $d$. \end{proof}

\subsection{A good increasing filtration on $\cS(V)$} Define $\cS(V)_{(r)}$ to be the linear span of the normally ordered monomials \begin{equation} \label{goodsv} \{:\partial^{k_1} \beta^{x_1} \cdots \partial^{k_s}\beta^{x_s}\partial^{l_1} \gamma^{y'_1}\cdots \partial^{l_t}\gamma^{y'_t}:|\ x_i\in V,\ y'_i\in V^*,\  k_i,l_i\geq 0,\  s+t \leq r\}.\end{equation} This is a good increasing filtration. We have $\cS(V) \cong \text{gr}(\cS(V))$ as linear spaces, and \begin{equation} \label{structureofgrs} \text{gr}(\cS(V))\cong \text{Sym} (\bigoplus_{k\geq 0} (V_k\oplus V^*_k)), \qquad V_k = \{\beta^{x}_k |\ x\in V\}, \qquad V^*_k = \{\gamma^{x'}_k |\  x'\in V^*\},\end{equation} as commutative algebras. Here $\beta^{x}_k$ and $\gamma^{x'}_k$ are the images of $\partial^k \beta^{x}$ and $\partial^k\gamma^{x'}$ in $\text{gr}(\cS(V))$ under the projection $\phi_1$ given by \eqref{projmap}. The map \eqref{deftheta} induces an action of $\gg[t]$ on $\text{gr}(\cS(V))$ by derivations of degree zero, defined on generators by 
\begin{equation}\label{actiontheta}\xi t^r(\beta^x_i) = \lambda^r_i \beta^{\rho(\xi)(x)}_{i-r}, \qquad \xi t^r(\gamma^{x'}_i) = \lambda^r_i \gamma^{\rho^*(\xi)(x')}_{i-r}, \qquad \lambda^r_i = \bigg\{ \begin{matrix} \frac{i!}{(i-r)!}  & 0\leq r\leq i \cr 0 & r>i \end{matrix}.\end{equation}
The derivation $\partial$ on $\text{gr}(\cS(V))$ is given by \begin{equation}\label{actionofpartial}\partial \beta^{x}_i = \beta^x_{i+1}, \qquad \partial \gamma^{x'}_i = \gamma^{x'}_{i+1},\end{equation} and since $[\partial, \xi t^r] = - r \xi t^{r-1}$, $\text{gr}(\cS(V))^{\gg[t]}$ is a closed under $\partial$. Finally, there is an injective map of $\partial$-rings \begin{equation}\label{injgamm} \text{gr}(\cS(V)^{\gg[t]})\hookrightarrow \text{gr}(\cS(V))^{\gg[t]},\end{equation} which is in general not surjective. Let $R$ be the image \eqref{injgamm}, and let $\{a_i| i\in I\}$ be a collection of generators for $R$ as a $\partial$-ring. By Lemma \ref{reconlem}, any set $\{a_i(z)\in \cS(V)^{\gg[t]}| i\in I\}$ such that $d_i = \text{deg}(a_i)$ and $\phi_{d_i}(a_i(z)) = a_i$, is a strong generating set for $\cS(V)^{\gg[t]}$.

\subsection{Arc spaces}
A connection between vertex algebras and arc spaces was observed in \cite{FBZ}, where the authors pointed out that for any affine variety $X$, the ring of polynomial functions on the arc space $X_{\infty}$ has the structure of an abelian vertex algebra. Conversely, the $\partial$-ring $\text{gr}(\cA)$ of a vertex algebra $\cA\in \cR$ can often be realized as the ring of polynomial functions $\cO(X_{\infty})$ for some $X$. For example, $\text{gr}(\cS(V))\cong \cO((V\oplus V^*)_{\infty})$. In our main examples, $\text{gr}(\cS(V)^{\gg[t]})\cong \cO(X_{\infty})$ where $X = \quot{(V\oplus V^*)}G = \text{Spec}(\cO(V\oplus V^*)^G)$ is the categorical quotient. More generally, whenever $\text{Spec}(\text{gr}(\cA))\cong X_{\infty}$ for some $X$, the geometry of $X_{\infty}$ encodes information about the vertex algebra structure of $\cA$.

We recall some basic facts about arc spaces, following the notation in \cite{EM}. Let $X$ be an irreducible scheme of finite type over $\mathbb{C}$. For each integer $m\geq 0$, the jet scheme $X_m$ is determined by its functor of points: for every $\mathbb{C}$-algebra $A$, we have a bijection
$$\text{Hom} (\text{Spec}  (A), X_m) \cong \text{Hom} (\text{Spec}  (A[t]/\langle t^{m+1}\rangle ), X).$$ Thus the $\mathbb{C}$-valued points of $X_m$ correspond to the $\mathbb{C}[t]/\langle t^{m+1}\rangle$-valued points of $X$. If $p>m$, we have projections $\pi_{p,m}: X_p \rightarrow X_m$ and $\pi_{p,m} \circ \pi_{q,p} = \pi_{q,m}$  when $q>p>m$. Clearly $X_0 = X$ and $X_1$ is the total tangent space $\text{Spec}(\text{Sym}(\Omega_{X/\mathbb{C}}))$. The assignment $X\mapsto X_m$ is functorial, and a morphism $f:X\ra Y$ induces $f_m: X_m \ra Y_m$ for all $m\geq 1$. If $X$ is nonsingular, $X_m$ is irreducible and nonsingular for all $m$.

If $X=\text{Spec}(R)$ where $R= \mathbb{C}[y_1,\dots,y_r] / \langle f_1,\dots, f_k\rangle$, we can find explicit equations for $X_m$. Define new variables $y_j^{(i)}$ for $i=0,\dots, m$, and define a derivation $D$ by $D(y_j^{(i)}) = y_j^{(i+1)}$ for $i<m$, and $D(y_j^{(m)}) =0$. This specifies the action of $D$ on all of $\mathbb{C}[y_1^{(0)},\dots, y_r^{(m)}]$; in particular, $f_{\ell}^{(i)} = D^i ( f_{\ell})$ is a well-defined polynomial in  $\mathbb{C}[y_1^{(0)},\dots, y_r^{(m)}]$. Letting $R_m = \mathbb{C}[y_1^{(0)},\dots, y_r^{(m)}] / \langle f_1^{(0)},\dots, f_k^{(m)}\rangle$, we have $X_m \cong \text{Spec} (R_m)$. By identifying $y_j$ with $y_j^{(0)}$, we see that $R$ is naturally a subalgebra of $R_m$. There is a $\mathbb{Z}_{\geq 0}$-grading $R_m = \bigoplus_{n\geq 0} R_m[n]$ by weight, defined by $\text{wt}(y^{(i)}_j) = i$. For all $m$, $R_m[0] = R$ and $R_m[n]$ is an $R$-module.

The {\it arc space} of $X$ is defined to be $$X_{\infty} = \lim_{\leftarrow} X_m.$$ If $X = \text{Spec}(R)$ as above, $X_{\infty}=\text{Spec}(R_{\infty})$ where $$R_{\infty}  = \mathbb{C}[y_1^{(0)},\dots,y_j^{(i)},\dots] / \bra f_1^{(0)},\dots, f_\ell^{(i)},\dots\ket.$$ Here $i=0,1,2,\dots$ and $D (y^{(i)}_j) = y^{(i+1)}_j$ for all $i$. By a theorem of Kolchin \cite{Kol}, $X_{\infty}$ is irreducible whenever $X$ is irreducible.

Let $G$ be a connected, reductive complex algebraic group with Lie algebra $\gg$. For $m\geq 1$, $G_m$ is an algebraic group which is the semidirect product of $G$ with a unipotent group $U_m$. The Lie algebra of $G_m$ is $\gg[t]/t^{m+1}$. Given a $G$-module $V$, there is an action of $G$ on $\cO(V)$ by automorphisms, and a compatible action of $\gg$ on $\cO(V)$ by derivations, satisfying $$\frac{d}{dt} \text{exp} (t\xi) (f)|_{t=0} = \xi(f), \qquad \xi\in\gg, \qquad f\in \cO(V).$$ Choose a basis $\{x_1,\dots,x_n\}$ for $V^*$, so that $$\cO(V) \cong  \mathbb{C}[x_1,\dots,x_n], \qquad \cO(V_m) =  \mathbb{C}[x_1^{(i)},\dots,x_n^{(i)}],\qquad 0\leq i\leq m.$$ Then $G_m$ acts on $V_m$, and the induced action of $\gg[t]/t^{m+1}$ on $\cO(V_m)$ is defined as follows. For $\xi \in\gg$, \begin{equation}\label{jetaction}\xi t^r (x_j^{(i)}) = \lambda^r_i (\xi(x_j))^{(i-r)}, \qquad \lambda^r_i = \bigg\{ \begin{matrix} \frac{i!}{(i-r)!}  & 0\leq r\leq i \cr 0 & r>i \end{matrix}.\end{equation} Via the projection $\gg[t]\ra \gg[t]/t^{m+1}$, $\gg[t]$ acts on $\cO(V_m)$. The invariant rings $\cO(V_m)^{\gg[t]}$ and $\cO(V_m)^{\gg[t]/t^{m+1}}$ coincide, and are equal to $\cO(V_m)^{G_m}$ since $G$ is connected.

In general, it is a very subtle problem to find generators for rings of the form $\cO(V_{m})^{G_{m}}$. The map $p:V\ra V /\!\!/ G$ induces a map \begin{equation} \label{arcmapm} p^*_m: \cO((\quot{V} G)_{m}) \rightarrow \cO(V_{m})^{G_{m}},\end{equation} which is neither surjective nor injective in general. Special cases of these rings were studied in \cite{E} and in the appendix of \cite{Mu}, and a more systematic study was carried out in \cite{LSS}. We call a $G$-module $V$ {\it coregular} if $V/\!\!/G$ is smooth; equivalently, $\cO(V)^G$ is a polynomial ring. We impose a mild technical condition which is automatic if $G$ is semisimple; we assume that $\cO(V)$ contains no nontrivial one-dimensional invariant subspaces. Equivalently, every semi-invariant of $G$ is invariant. The following result appears as Corollary 3.20 of \cite{LSS}.
 
\begin{thm} \label{jetthm} Let $G$ be a connected, reductive group, and let $V$ be a coregular $G$-module such that $\cO(V)$ contains no nontrivial one-dimensional $G$-invariant subspaces. Then \eqref{arcmapm} is an isomorphism for all $m$.
\end{thm}

It is immediate that \eqref{arcmapm} is an isomorphism for $m=\infty$ as well. There are several other examples in \cite{LSS} where \eqref{arcmapm} is an isomorphism for $m=\infty$. This holds when $G$ is one of the classical groups $\text{SL}_n$, $\text{GL}_n$, $\text{SO}_n$, or $\text{Sp}_{2n}$ and $V$ is a sum of copies of the standard representation of $G$ (and its dual in the case of $\text{SL}_n$ or $\text{GL}_n$), such that $V /\!\!/ G$ is a complete intersection. It holds when $G = \mathbb{C}^*$ and $V$ is a representation whose weights are all $\pm 1$. Finally, it holds when $G =  \text{SL}_2$ and $V$ is the sum of an arbitrary finite number of copies of the standard representation. 

The relevance of invariant rings of the form $\cO(V_{\infty})^{G_{\infty}}$ to our vertex algebra commutant problem is as follows. By \eqref{actiontheta} and \eqref{jetaction}, the map $$\Phi: \text{gr}(\cS(V)) \ra \cO((V\oplus V^*)_{\infty}),\qquad \beta^{x}_k \mapsto x^{(k)}, \qquad \gamma^{x'}_k \mapsto (x')^{(k)},$$ is an isomorphism of $\gg[t]$-algebras. Moreover, $\Phi^{-1} \circ  D\circ \Phi = \partial$, so we have an isomorphism of differential graded algebras \begin{equation} \label{dgaiso}\text{gr}(\cS(V))^{\gg[t]} \cong  \cO((V\oplus V^*)_{\infty})^{\gg[t]} = \cO((V\oplus V^*)_{\infty})^{G_{\infty}}.\end{equation} 
In general, the map \begin{equation} \label{arcmap} \cO((\quot{(V\oplus V^*)} G)_{\infty}) \rightarrow \cO((V\oplus V^*)_{\infty})^{G_{\infty}},\end{equation} is not an isomorphism, and even when it is, we may not be able to reconstruct $\cS(V)^{\gg[t]}$ because \eqref{injgamm} need not be surjective. We will see that in our main examples, both \eqref{injgamm} and \eqref{arcmap} are isomorphisms. Since $G$ is reductive, we may choose a finite set $\{f_1,\dots, f_r\}$ of generators for $\cO(V \oplus V^*)^G$ of homogeneous degrees $d_1,\dots, d_r$. Since \eqref{arcmap} is an isomorphism, $f_1,\dots, f_r$ correspond to generators of $\text{gr}(\cS(V))^{\gg[t]}$ as a differential algebra, and since \eqref{injgamm} is an isomorphism, we can find vertex operators $f_1(z),\dots, f_r(z) \in \cS(V)^{\gg[t]}$ satisfying $\phi_{d_i}(f_i(z)) = f_i$. By Lemma \ref{reconlem}, this is a strong generating set for $\cS(V)^{\gg[t]}$. Similarly, let $\{r_1,\dots, r_t\}$ be a set of generators for the ideal of relations among $f_1,\dots, f_r$. As in Lemma \ref{idealrecon}, there exist normally ordered polynomial relations $r_j(z)$, which are quantum corrections of $r_j$, for $j=1,\dots, n$.

\begin{thm} \label{minrelations} Suppose that both \eqref{injgamm} and \eqref{arcmap} are isomorphisms. If $\{f_1,\dots, f_r\}$ is a minimal generating set for $\cO(V \oplus V^*)^G$,  $\{f_1(z), \dots,f_r(z)\}$ is a minimal strong generating set for $\cS(V)^{\gg[t]}$ as a vertex algebra. Moreover, all normally ordered polynomial relations among $f_1(z), \dots, f_r(z)$ and their derivatives are consequences of $r_1(z),\dots, r_t(z)$ in the sense of Lemma \ref{idealrecon}. \end{thm}

\begin{proof} The first statement is clear from Lemma \ref{reconlem}. Since \eqref{arcmap} is an isomorphism, it follows that $r_1,\dots, r_t$ generate the ideal of relations in $\cO((V\oplus V^*)_{\infty})^{G_{\infty}}$ as a differential ideal. The second statement then follows from Lemma \ref{idealrecon}.
\end{proof}

Finally, when both \eqref{injgamm} and \eqref{arcmap} are isomorphisms, our commutative diagram
$$\begin{array}[c]{ccc} \cS(V)^{\gg[t]}&\stackrel{}{\rightarrow}& \cS(V)  \\ \downarrow\scriptstyle{\pi}&&\downarrow\scriptstyle{\pi_{\text{Zhu}}}\\ \cD(V)^{G} &\stackrel{}{\rightarrow}& \cD(V) \end{array}$$ has the following nice property.

\begin{thm} \label{zhudescription} Suppose that \eqref{injgamm} and \eqref{arcmap} are isomorphisms. Then the Zhu algebra $A(\cS(V)^{\gg[t]})$ is isomorphic to $\cD(V)^G$ and the map $\pi$ above is surjective. \end{thm}

\begin{proof} Fix a generating set $\{f_1,\dots, f_r\}$ for $\cD(V)^G$ such that $f_i$ has degree $d_i$, which under the above hypotheses corresponds to a strong generating set $\{f_1(z),\dots, f_r(z)\}$ for $\cS(V)^{\gg[t]}$ as a vertex algebra. Without loss of generality, we may assume that $f_i(z) \in \cS(V)^{\gg[t]}_{(d_i)} \setminus  \cS(V)^{\gg[t]}_{(d_i-1)}$. Then $A(\cS(V)^{\gg[t]})$ is generated by $\{\tilde{f}_1,\dots, \tilde{f}_r\}$ where $\tilde{f}_i = \pi_{\text{Zhu}}(f_i(z))$. 

Clearly $\pi(f_i(z)) = f_i$ up to corrections of lower degree in the Bernstein filtration, so by induction on degree we see that $\pi$ is surjective. The inclusion $\cS(V)^{\gg[t]} \hookrightarrow \cS(V)$ induces a map of Zhu algebras $h:A(\cS(V)^{\gg[t]}) \rightarrow A(\cS(V)) = \cD(V)$ whose image is clearly $\cD(V)^G$ since $h(\tilde{f}_i) =f_i$ up to lower order corrections. We claim that $h$ has trivial kernel, and is therefore an isomorphism from $A(\cS(V)^{\gg[t]})$ to $\cD(V)^G$. Let $r\in \text{Ker}(h)$ be an element of degree $d$, regarded as a polynomial among the $\tilde{f}_i$'s. Under $h$, it maps to a relation in $\cD(V)^G$ whose leading term is the same, with $\tilde{f}_i$ replaced by $f_i$. By Theorem \ref{minrelations}, there is an analogous relation $r(z)\in \cS(V)^{\gg[t]}$, which is obtained up to lower order correction by replacing each $f_i$ with $f_i(z)$ and replacing ordinary products with Wick products. Since $r(z)$ is identically zero, and $\pi_{\text{Zhu}}(r(z)) \in A(\cS(V)^{\gg[t]}$ has the same leading term as $r$, the claim follows by induction on degree. \end{proof}

\section{The cases $G=\text{SL}_n$ and $G=\text{GL}_n$}\label{sectionsln}

For $n\geq 1$, let $V$ be the direct sum of $m$ copies of $\mathbb{C}^n$, with basis $\{x_{1,j},\dots, x_{n,j}|\ j=1,\dots, m\}$, which is just the space of $n\times m$ matrices. The left action of $\text{GL}_n$ and right action of $\text{GL}_m$ on $V$ induce actions of $\text{GL}_n$ and $\text{GL}_m$ on $\cD(V)$, infinitesimal actions $\gg\gl_n\ra \text{Der}(\cD(V))$ and $\gg\gl_m\ra \text{Der}(\cD(V))$, and algebra homomorphisms
\begin{equation}\label{taugln} \tau: U( \gg\gl_n) \ra \cD(V), \qquad  \tau': U (\gg\gl_m) \ra \cD(V).\end{equation} By classical $\text{GL}_n$ -$\text{GL}_m$ Howe duality, $\cD(V)^{\text{GL}_n} = \tau'(U( \gg\gl_m))$ and $\cD(V)^{\text{GL}_m} = \tau(U (\gg\gl_n))$. Moreover, $\cD(V)^{\text{GL}_n}$ and $\cD(V)^{\text{GL}_m}$ form a pair of mutual commutants inside $\cD(V)$.

Next we consider $\cD(V)^{\text{SL}_n}$ for $n\geq 2$. If $m<n$, $\cD(V)^{\text{SL}_n} = \tau'(U (\gg\gl_m))$, but $\cD(V)^{\text{SL}_n}$ and $\tau(U( \gs\gl_n))$ are not mutual commutants because $$\text{Com}(\cD(V)^{\text{SL}_n},\cD(V)) = \text{Com}(\tau'(U (\gg\gl_m)),\cD(V)) = \cD(V)^{\text{GL}_m}= \tau(U( \gg\gl_n)).$$ For $m\geq n$, let $J = \{j_1,\dots, j_n\}\subset \{1,\dots, m\}$ be a set of distinct indices, and let 
$$d_{J} = \det \left[\begin{matrix} x_{1,j_1} & \cdots & x_{1,j_n} \cr  \vdots  & & \vdots  \cr  x_{n,j_1}  & \cdots & x_{n,j_n} \end{matrix} \right], \qquad d'_{J} = \det \left[\begin{matrix} x'_{1,j_1} & \cdots & x'_{1,j_n} \cr  \vdots  & & \vdots  \cr  x'_{n,j_1}  & \cdots & x'_{n,j_n} \end{matrix} \right].$$
Then $\cD(V)^{\text{SL}_n}$ is generated by $\tau'(U(\gg\gl_m))$ together with the $d_J$ and $d'_J$, which clearly do not lie in $\tau'(U( \gg\gl_m))$ \cite{We}.

At the vertex algebra level, the induced maps
$$\hat{\tau}: V_{-m}(\gg\gl_n) \ra \cS(V), \qquad \hat{\tau}': V_{-n}(\gg\gl_{m}) \ra \cS(V)$$ corresponding to (\ref{taugln}) are given by (\ref{deftheta}). Let \begin{equation}\label{thetaprime} \Theta = \hat{\tau}(V_{-m}(\gg\gl_n)), \qquad \Theta'  = \hat{\tau}'(V_{-n}(\gg\gl_{m})).\end{equation} For $n,m\geq 2$, we can restrict $\hat{\tau}$ and $\hat{\tau}'$ to the subalgebras $V_{-m}(\gs\gl_n)$ and $V_{-n}(\gs\gl_{m})$. Let $$\bar{\Theta} = \hat{\tau}(V_{-m}(\gs\gl_{n})), \qquad \bar{\Theta}' = \hat{\tau}'(V_{-n}(\gs\gl_{m})).$$ We have $\Theta = \bar{\Theta}\otimes \cH$ and $\Theta' =  \bar{\Theta}'\otimes \cH$, where $\cH$ is a copy of the Heisenberg vertex algebra corresponding to the center of both $\gg\gl_n$ and $\gg\gl_m$. The generator of $\cH$ is \begin{equation}\label{defofthetah} e = \sum_{i=1}^n\sum_{j=1}^m :\gamma^{x'_{i,j}}\beta^{x_{i,j}}:,\end{equation} which corresponds to the Euler operator $\sum_{i=1}^n\sum_{j=1}^m x'_{i,j} \frac{\partial}{\partial x'_{i,j}}$ and satisfies the OPE relation $e(z)e(w) \sim -mn (z-w)^{-2}$

We will assume that $n\geq 2$ for the rest of this section, since the case $n=1$ is a special case of Theorem 7.3 of \cite{L}. First we consider $\cS(V)^{\gs\gl_n[t]}$.
\begin{thm} \label{slncasei} Let $1\leq m<n$. Then $$\cS(V)^{\gs\gl_n[t]} = \Theta'.$$ In particular, $\cS(V)^{\gs\gl_n[t]}$ has a minimal strong finite generating set $$\{\theta^{e_{a,b}} = \hat{\tau}' (X^{e_{a,b}})|\ 1\leq a \leq m,\ 1\leq b\leq m\},$$ where $\{e_{a,b}\}$ is the standard basis for $\gg\gl_m$. There are no normally ordered polynomial relations among the $\theta^{e_{a,b}}$'s and their derivatives, so $\hat{\tau}'$ is injective and $\Theta' \cong V_{-n}(\gg\gl_m)$. \end{thm}

\begin{proof} The generators $\{\tau'(e_{a,b})\}$ of $\cD(V)^{\text{SL}_n}$ correspond to generators of $\cO(V\oplus V^*)^{\text{SL}_n}$, which by abuse of notation we also denote by $\tau'(e_{a,b})$. For $m<n$, Weyl's second fundamental theorem of invariant theory for $\text{SL}_n$ shows that there are no relations among the $\tau'(e_{a,b})$'s, so $V\oplus V^*$ is coregular. The remaining hypotheses of Theorem \ref{jetthm} are also satisfied, so these elements generate $\cO((V\oplus V^*)_{\infty})^{\gs\gl_n[t]} \cong \text{gr}(\cS(V))^{\gs\gl_n[t]}$ as a differential algebra. Finally, the vertex operators $\theta^{e_{a,b}}$ corresponding to $\tau'(e_{a,b})$ lie in $\cS(V)^{\gs\gl_n[t]}$, so \eqref{injgamm} is surjective, and hence is an isomorphism. The claim then follows from Theorem \ref{minrelations}. \end{proof}

For $m\geq n$, the elements $d_J, d'_J\in\cD(V)^{\text{SL}_n}$ correspond to vertex operators 
$$D_{J} = \det \left[\begin{matrix} \beta^{x_{1,j_1}} & \cdots &  \beta^{x_{1,j_n}} \cr  \vdots  & & \vdots  \cr   \beta^{x_{n,j_1}}  & \cdots &  \beta^{x_{n,j_n}} \end{matrix} \right], \qquad D'_{J} = \det \left[\begin{matrix} \gamma^{x'_{1,j_1}} & \cdots & \gamma^{x'_{1,j_n}} \cr  \vdots  & & \vdots  \cr  \gamma^{x'_{n,j_1}}  & \cdots & \gamma^{x'_{n,j_n}} \end{matrix} \right],$$
which satisfy $\pi_{\text{Zhu}}(D_J) = d_J$ and $\pi_{\text{Zhu}}(D'_J) = d'_J$. Both $D_J$ and $D'_J$ lie in $\cS(V)^{\gs\gl_n[t]}$; this is clear because the generators $\{\theta^{\xi}| \xi\in \gs\gl_n\}$ of $\bar{\Theta}$ are of the form (\ref{deftheta}), and the nonnegative circle products of $\theta^{\xi}$ with $D_J$ or $D'_J$, which have either only $\beta$'s or $\gamma$'s, can have no double contractions. 

In the case $m=n$, $V\oplus V^*$ is not coregular but $(V\oplus V^*)/\!\!/ \text{SL}_n$ is a hypersurface, and \eqref{arcmap} is an isomorphism by Theorem 4.8 of \cite{LSS}. In this case there are only two determinants $d$ and $d'$, and the ideal of relations in $\cD(V)^{\text{SL}_n}$ is generated by a single relation whose leading term is $dd' =  \text{det}[\tau'(e_{a,b})]$ for $a,b = 1,\dots, n$.  

\begin{thm} For $m=n$, $\cS(V)^{\gs\gl_n[t]}$ is strongly generated as a vertex algebra by $\Theta'$ together with $D$ and $D'$. Moreover, all normally ordered polynomial relations are consequences of the above classical relation in the sense of Lemma \ref{idealrecon}. \end{thm}

In the case $m>n$, $(V\oplus V^*)/\!\!/ \text{SL}_n$ is not a complete intersection, so we cannot conclude that $\cS(V)^{\gs\gl_n[t]}$ is strongly generated by $\Theta'$ together with the $D_J$ and $D'_{J}$. However, in the case $n=2$, the representations $\mathbb{C}^2$ and $(\mathbb{C}^2)^*$ of $\text{SL}_2$ are isomorphic, and since \eqref{arcmap} is an isomorphism in this case, we can describe $\cS(V)^{\gs\gl_2[t]}$ for all $m$.

\begin{thm} \label{sl2caseii} For $n=2$ and $m\geq 2$, $\cS(V)^{\gs\gl_2[t]}$ is strongly generated as a vertex algebra by $\Theta'$ together with $\{D_J,D'_{J}\}$ where $J$ runs over all $2$-element subsets of $\{1,\dots, m\}$. All normally ordered polynomial relations among the generators and their derivatives come from Weyl's second fundamental theorem of invariant theory for the standard representation of $\text{SL}_2$, in the sense of Lemma \ref{idealrecon}.\end{thm}

\begin{proof} Since $\cD(V)^{\text{SL}_2}$ is generated by $\{\tau'(\eta)| \eta'\in \gg\gl_m\}$ together with the determinants $\{d_{J},d'_{J}\}$, $\text{gr}(\cS(V))^{\gs\gl_2[t]}$ is generated as a $\partial$-ring by the corresponding set, under the isomorphism $\text{gr}(\cD(V)) \cong \cO(V\oplus V^*)^{\text{SL}_2}$. Since the corresponding vertex operators $\theta^{\eta}$, $D_J$, and $D'_{J}$ all lie in $\cS(V)^{\gs\gl_2[t]}$, the claim follows. \end{proof}

Next, for $m<n$ we can use our description of $\cS(V)^{\gs\gl_n[t]}$ to find strong generators for $\cS(V)^{\gg\gl_n[t]}$. 

\begin{thm} \label{glncase} For $m=1$, $\cS(V)^{\gg\gl_n[t]} = \mathbb{C}$. For $2\leq m< n$, $\cS(V)^{\gg\gl_n[t]}  = \bar{\Theta}'$.\end{thm}

\begin{proof} Since $\Theta = \bar{\Theta}\otimes \cH$, we have $$\cS(V)^{\gg\gl_n[t]} = \text{Com}(\bar{\Theta}\otimes \cH,\cS(V)) = \text{Com}( \cH, \text{Com}(\bar{\Theta},\cS(V))) = \text{Com}(\cH, \Theta') .$$ For $m=1$, $\Theta' = \cH$, so $\cS(V)^{\gg\gl_n[t]} = \text{Com}(\cH,\cH) =  \mathbb{C}$. For $m>1$, we have $\Theta' = \bar{\Theta}'\otimes \cH$, so $$\cS(V)^{\gg\gl_n[t]}= \text{Com}(\cH, \bar{\Theta}'\otimes \cH) = \bar{\Theta}'.$$  \end{proof}

Unfortunately, even in the case $n=2$, we cannot describe $\cS(V)^{\gg\gl_2[t]}$ for $m\geq 2$ using these methods. For all $n\geq 2$ and $m\geq n$, $\cS(V)^{\gg\gl_n[t]}$ is strictly larger than $\bar{\Theta}'$, and does not have a simple description in terms of the generators for $\cD(V)^{\text{GL}_n}$. To see this, let $\cA$ be the subalgebra of $\cS(V)^{\gs\gl_n[t]}$ of $\beta\gamma$-charge zero, which is just the subalgebra of $\cS(V)^{\gs\gl_n[t]}$ annihilated by the zero mode $e(0) = e\circ_0$. Since $\partial^k D_J$ and $\partial^k D'_J$ have eigenvalues $-n$ and $n$, respectively, under $e(0)$ for all $J$ and $k$, $\cA$ contains $\Theta'$ together with the collection $$\{:\partial^k D_J \partial^l D'_{J'}:|  k,l\geq 0\}.$$ Next, we claim that $\Theta'$ is a {\it proper} subalgebra of $\cA$. For $r\geq 2$, let $$\Theta'_{(r)} = \Theta'\cap \cS(V)_{(r)}, \qquad \cA_{(r)} = \cA\cap \cS(V)_{(r)}.$$ Since $\Theta' = \bar{\Theta}' \otimes \cH$, and $e$ generates $\cH$, it follows that for $k>0$, $e\circ_k$ maps $\Theta'_{(r)}$ into $\Theta'_{(r-2)}$. However, note that $:D_J \partial D'_{J'}: \in \cA_{(2n)}$ and $$e\circ_1 \big(:D_J \partial D'_{J'}:\big) =\  :D_J D'_{J'}: -n D_J \circ_0 D'_{J'},$$ which does not lie in $\cA_{(2n-2)}$ since $:D_J D'_{J'}:$ has degree $2n$, whereas $D_J \circ_0 D'_{J'} \in \cA_{(2n-2)}$.  Hence $:D_J \partial D'_{J'}:\notin \Theta'$, as claimed. Finally, $\cA$ and $\Theta'$ are both modules over $\cH$, and decompose as direct sums of $\cH$-modules of highest weight zero (i.e., $e(0)$ acts by zero). Since $\bar{\Theta}' = \text{Com}(\cH, \Theta')$ and $\cS(V)^{\gg\gl_n[t]} = \text{Com}(\cH,\cA)$, $\bar{\Theta}'$ must be a proper subalgebra of $\cS(V)^{\gg\gl_n[t]}$.

We now consider the double commutants $\text{Com}(\cS(V)^{\gg\gl_n[t]},\cS(V))$ and $\text{Com}(\cS(V)^{\gs\gl_n[t]},\cS(V))$. 

\begin{thm} \label{howegln} For $m>n$, $\text{Com}(\cS(V)^{\gg\gl_n[t]}, \cS(V)) = \Theta$, so $\Theta$ and $\cS(V)^{\gg\gl_n[t]}$ form a Howe pair inside $\cS(V)$. For $m < n$, $\Theta$ is not a member of a Howe pair.\end{thm}

\begin{proof} First, let $m>n$. Even though we do not have a description of $\cS(V)^{\gg\gl_n[t]}$ in this case, we have $\bar{ \Theta}' \subset \cS(V)^{\gg\gl_n[t]}$, so 
\begin{equation}\label{inclusions} \Theta\subset \text{Com}(\cS(V)^{\gg\gl_n[t]},\cS(V)) \subset \text{Com}(\bar{\Theta}',\cS(V)) = \cS(V)^{\gs\gl_m[t]}.\end{equation} By Theorem \ref{slncasei} (applied now to the right action of $\gs\gl_m$), we have $\cS(V)^{\gs\gl_m[t]} = \Theta$. Hence the inclusions in (\ref{inclusions}) are equalities, so $\Theta$ and $\cS(V)^{\gg\gl_n[t]}$ form a Howe pair.

Next, suppose that $m <n$. If $m=1$, $\text{Com}(\cS(V)^{\gg\gl_n[t]},\cS(V)) = \text{Com}(\C,\cS(V)) = \cS(V)$, which is larger than $\Theta$. If $2\leq m < n$, $\cS(V)^{\gg\gl_n[t]} = \bar{\Theta}'$, but $\text{Com}(\cS(V)^{\gg\gl_n[t]},\cS(V)) = \cS(V)^{\gs\gl_m[t]}$ contains $\Theta$ together with vertex operators $$D_{I} = \det \left[\begin{matrix} \beta^{x_{i_1,1}} & \cdots &  \beta^{x_{i_1,m}} \cr  \vdots  & & \vdots  \cr   \beta^{x_{i_m,1}}  & \cdots &  \beta^{x_{i_m,m}} \end{matrix} \right], \qquad D'_{I} = \det \left[\begin{matrix} \gamma^{x'_{i_1,1}} & \cdots & \gamma^{x'_{i_1,m}} \cr  \vdots  & & \vdots  \cr  \gamma^{x'_{i_m,1}}  & \cdots & \gamma^{x'_{i_m,m}} \end{matrix} \right],$$ where $I = \{i_1,\dots, i_m\}\subset \{1,\dots,n\}$ satisfies $1\leq i_1< \cdots < i_m\leq n$. So $\text{Com}(\cS(V)^{\gg\gl_n[t]},\cS(V))$ is larger than $\Theta$, since $\Theta$ is homogeneous of $\beta\gamma$-charge zero, but $D_I$ and $D'_I$ have $\beta\gamma$-charges $-m$ and $m$, respectively. \end{proof}

\begin{thm} \label{howesln} For $m>n$, $\text{Com}(\cS(V)^{\gs\gl_n[t]}, \cS(V)) = \bar{\Theta}$, so $\bar{\Theta}$ and $\cS(V)^{\gs\gl_n[t]}$ form a Howe pair. For $2\leq m <n$, $\bar{\Theta}$ is not a member of a Howe pair. For $m=1$, $\bar{\Theta}$ is a member of a Howe pair for $n\geq 3$, but not for $n=2$. \end{thm}

\begin{proof} For $m>n$, $\Theta'  \subset \cS(V)^{\gs\gl_n[t]}$ and $\text{Com}(\Theta', \cS(V)) = \cS(V)^{\gg\gl_m[t]} = \bar{\Theta}$, so the claim follows. If $2 \leq m <n$, $\text{Com}(\cS(V)^{\gs\gl_n[t]}, \cS(V)) = \cS(V)^{\gg\gl_m[t]}$, which is larger than $\bar{\Theta}$ for the same reason that $\cS(V)^{\gg\gl_n[t]}$ is larger than $\bar{\Theta}'$ for $m>n$. For $m=1$, $\cS(V)^{\gs\gl_n[t]} = \cH$, so we can compute $\text{Com}(\cS(V)^{\gs\gl_n[t]}, \cS(V)) = \text{Com}(\cH, \cS(V))$ using Theorem 7.3 of \cite{L}. In particular, $\text{Com}(\cH, \cS(V))$ contains $n$ commuting copies of the (simple) Zamolodchikov $\cW_3$ algebra with central charge $c=-2$. For $n=2$, a calculation shows that these copies of $\cW_3$ do not lie in $\bar{\Theta}$, so $\text{Com}(\cH, \cS(V))$ is larger than $\bar{\Theta}$. For $n\geq 3$, the fact that $\bar{\Theta}$ is a member of a Howe pair follows from Theorem 3.1 of \cite{AP}. \end{proof}

\section{The case $G = \text{SO}_n$}
For $n\geq 3$, let $G = \text{SO}_n$ and let
$V$ be the direct sum of $m$ copies of $\mathbb{C}^n$, with basis $\{x_{1,j},\dots, x_{n,j}|\ j=1,\dots,m\}$. The action of $\text{SO}_n$ on $V$ induces an action of $\gs\go_n$ on $\cD(V)$ and an algebra homomorphism $\tau: U(\gs\go_n) \ra \cD(V)$. There is a well-known Lie algebra homomorphism $\tau':\gs\gp_{2m}\ra \cD(V)$, which appears on p. 219-220 of \cite{GW}. First, $\gs\gp_{2m}$ is the subset of $\gg\gl_{2m}$ consisting of block matrices of the form $$\bigg[ \begin{matrix}  A & B \cr  C & -A^T \end{matrix} \bigg], \qquad A,B,C\in \gg\gl_m, \qquad B=B^T, \qquad C = C^T.$$ In terms of the basis $\{e_{i,j}| 1\leq i\leq 2m, \ 1\leq j\leq 2m\}$ for $\gg\gl_{2m}$, a standard basis for $\gs\gp_{2m}$ consists of $$e_{j,k+m} + e_{k,j+m}, \qquad  -e_{j+m,k} - e_{k+m,j}, \qquad e_{j,k} - e_{m+k, m+j}, \qquad 1\leq j,k\leq m.$$ Define the following elements of $\cD(V)$ $$ M_{j,k} =  \sum_{i=1}^n x'_{i,j} x'_{i,k}, \qquad \Delta_{j,k} = \sum_{i=1}^n \frac{\partial}{\partial x'_{i,j}} \frac{\partial}{\partial x'_{i,k}}, \qquad  E_{j,k} =  \sum_{i=1}^n x'_{i,j} \frac{\partial}{\partial x'_{i,k}},$$ and define the homomorphism $\tau':\gs\gp_{2m}\ra \cD(V)$ by
\begin{equation}\label{explicitson} e_{j,k+m} + e_{k,j+m} \mapsto M_{j,k} , \qquad  -e_{j+m,k} - e_{k+m,j} \mapsto \Delta_{j,k}, \qquad  e_{j,k} - e_{m+k, m+j} \mapsto E_{j,k} + \frac{n}{2}\delta_{j,k},\end{equation} which extends to an algebra homomorphism $\tau': U(\gs\gp_{2m}) \ra \cD(V)$. We obtain an action $\gs\gp_{2m}\ra \text{End}(\cD(V))$ by derivations of degree zero, where $\eta\in \gs\gp_{2m}$ acts by $[\tau'(\eta),-]$. As an $\gs\gp_{2m}$-module, $\text{gr}(\cD(V))\cong \text{Sym}\big(\bigoplus_{i=1}^n U_i\big)$, where $U_i$ is the copy of the standard $\gs\gp_{2m}$-module $\mathbb{C}^{2m}$ with basis $\{x'_{i,1},\dots, x'_{i,m}, \frac{\partial}{\partial x'_{i,1}},\dots, \frac{\partial}{\partial x'_{i,m}}\}$. For $1\leq m< \frac{n}{2}$, $\cD(V)^{\text{SO}_n} = \tau'(U(\gs\gp_{2m}))$, and for $m\geq \frac{n}{2}$, $\cD(V)^{\text{SO}_n}$ is generated by $\tau'(U(\gs\gp_{2m}))$ together with the $n\times n$ determinants $d_J$. Here $J = \{j_1,\dots, j_n\}$ is a set of indices in $\{1,\dots, m\}$, and each index corresponds to either the linear functions or the differential operators. For any $m\geq 1$, $\cD(V)^{\gs\gp_{2m}} = \tau(U(\gs\go_n))$, so $\cD(V)^{\text{SO}_n}$ and $ \tau(U(\gs\go_n))$ form a pair of mutual commutants inside $\cD(V)$. This is known as $\text{SO}_n$ - $\gs\gp_{2m}$ Howe duality.

The vertex algebra analogues of $\tau$ and $\tau'$ are 
\begin{equation}\label{hattauva}\hat{\tau}: V_{-2m}(\gs\go_n) \ra \cS(V), \qquad \hat{\tau}': V_{-\frac{n}{2}}(\gs\gp_{2m}) \ra \cS(V).\end{equation} The map $\hat{\tau}$ is given by (\ref{deftheta}), and $\hat{\tau}'$ is given by $$ X^{e_{j,k+m} + e_{k,j+m}} \mapsto \sum_{i=1}^n : \gamma^{x'_{i,j}} \gamma^{x'_{i,k}}:, \qquad X^{-e_{j+m,k} - e_{k+m,j}} \mapsto \sum_{i=1}^n :\beta^{x_{i,j}}\beta^{x_{i,k}}:,$$ $$ X^{e_{j,k} - e_{m+k, m+j}} \mapsto \sum_{i=1}^n :\gamma^{x'_{i,j}} \beta^{x'_{i,k}}:. $$ As usual, let $\Theta = \hat{\tau}(V_{-2m}(\gs\go_n))$ and $\Theta' = \hat{\tau}'( V_{-\frac{n}{2}}(\gs\gp_{2m}))$. 

\begin{thm}\label{howeson} For $m<\frac{n}{2}$, we have $\cS(V)^{\gs\go_n[t]} = \Theta'$, and $\Theta' \cong V_{-\frac{n}{2}}(\gs\gp_{2m}))$. For $m \geq  \frac{n}{2}-1$, $\Theta$ and $\cS(V)^{\gs\go_n[t]}$ form a Howe pair inside $\cS(V)$.\end{thm}

\begin{proof} For $m<\frac{n}{2}$, $V\oplus V^*$ is coregular, so the generators of $\text{gr}(\cS(V))^{\gs\go_n[t]}$ as a $\partial$-ring correspond to the generators of $\cO(V\oplus V^*)^{\text{SO}_n}$. By Weyl's first fundamental theorem of invariant theory for $\text{SO}_n$, these generators are quadratic and correspond to the generators of $\Theta'$. An OPE calculation shows that $\Theta'\subset \cS(V)^{\gs\go_n[t]}$. It follows that the map \eqref{injgamm} is surjective, so $\cS(V)^{\gs\go_n[t]}= \Theta'$. The second statement is immediate from Weyl's second fundamental theorem of invariant theory for $\text{SO}_n$.

For $m \geq \frac{n-1}{2}$, $\bigoplus_{i=1}^n U_i$ is coregular as an $\Sp_{2m}$-module, and when $n$ is even and $m = \frac{n}{2}-1$, $(\bigoplus_{i=1}^n U_i)/\!\!/ \text{Sp}_{2m}$ is a hypersurface. By Theorem \ref{jetthm} in the case $m \geq \frac{n-1}{2}$, and by Theorem 4.5 of \cite{LSS} in the case $m = \frac{n}{2}-1$, the map $$p^*_{\infty}: \cO(((\bigoplus_{i=1}^n U_i)/\!\!/  \text{Sp}_{2m})_{\infty}) \ra \cO((\bigoplus_{i=1}^n U_i)_{\infty})^{(\text{Sp}_{2m})_{\infty}}$$ is an isomorphism. Therefore the generators of $\text{gr}(\cS(V))^{\gs\gp_{2m}[t]}$ correspond to the generators of $\cO(\bigoplus_{i=1}^n U_i)^{\text{Sp}_{2m}}$.

By Weyl's first fundamental theorem of invariant theory for $\text{Sp}_{2m}$ these generators are quadratic and correspond to the generators of $\Theta$. Since $\Theta\subset \cS(V)^{\gs\gp_{2m}[t]} =  \text{Com}(\Theta',\cS(V))$, the map $\text{gr}(\cS(V)^{\gs\gp_{2m}[t]})\hookrightarrow \text{gr}(\cS(V))^{\gs\gp_{2m}[t]}$ is surjective, and $\cS(V)^{\gs\gp_{2m}[t]} = \Theta$. Finally, since $\Theta'\subset \cS(V)^{\gs\go_n[t]}$, it follows that $\text{Com}(\cS(V)^{\gs\go_n[t]}, \cS(V)) = \text{Com}(\Theta',\cS(V)) = \Theta$. \end{proof}

\begin{remark} Since $\gs\gl_2\cong \gs\go_3$ as complex Lie algebras, and the adjoint representation of $\gs\gl_2$ coincides with the standard representation of $\gs\go_3$, we recover the main result (Theorem 1.3) of \cite{LL} by taking $n=3$ and $m=1$ in the preceding theorem.\end{remark}

For $n$ even and $m =\frac{n}{2}$, $(V\oplus V^*)/\!\!/ \text{SO}_n$ is a hypersurface, and \eqref{arcmap} is an isomorphism by Theorem 4.7 of \cite{LSS}. However, we cannot use this result to reconstruct $\cS(V)^{\gs\go_n[t]}$ because \eqref{injgamm} is not surjective. For example, in the case $n=4$ and $m=2$, consider the $4\times 4$ determinant $d$ corresponding to the modules $W_1, W^*_1$, $W_2$, $W^*_2$ which is a second-order differential operator of degree four in the Bernstein filtration. Let $D\in \cS(V)$ be any vertex operator obtained from $d$ by replacing $x'_{i,j}$ and $\frac{\partial}{\partial x'_{i,j}}$ with $\gamma^{x'_{i,j}}$ and $\beta^{x_{i,j}}$, respectively, and by replacing ordinary products with Wick products. A calculation shows that for any vertex operator $D'\in \cS(V)$ of lower degree, $D + D'\notin \cS(V)^{\gs\go_4[t]}$, so \eqref{injgamm} is not surjective.

\section{The case $G = \text{Sp}_{2n}$}
For $n\geq 2$, let $G = \text{Sp}_{2n}$ and let $V$ be the direct sum of $m$ copies of  $\mathbb{C}^{2n}$, with basis $\{x_{1,j},\dots, x_{2n,j}|\ j=1,\dots,m\}$. The action of $\text{Sp}_{2n}$ on $V$ induces an action of $\gs\gp_{2n}$ on $\cD(V)$, and an algebra homomorphism
$\tau: U( \gs\gp_{2n}) \ra \cD(V)$. If $m\geq 2$, there is a Lie algebra homomorphism $\tau': \gs\go_{2m}\ra \cD(V)$, which appears on p. 221-222 of \cite{GW}. We regard $\gs\go_{2m}$ as the subset of $\gg\gl_{2m}$ consisting of block matrices of the form $$\bigg[ \begin{matrix}  A & B \cr C & -A^T\end{matrix} \bigg], \qquad A,B,C\in \gg\gl_m, \qquad B=-B^T, \qquad C = -C^T,$$ with basis $$e_{j,k+m} - e_{k,j+m}, \qquad e_{j+m,k} - e_{k+m,j}, \qquad e_{j,k} - e_{m+k, m+j}, \qquad 1\leq j,k\leq m.$$ Define the following elements of $\cD(V)$
$$D_{j,k} = \sum_{i=1}^n \frac{\partial}{\partial x'_{i,j}} \frac{\partial}{\partial x'_{i+n,k}} - \frac{\partial}{\partial x'_{i+n,j}} \frac{\partial}{\partial x'_{i,k}}, \qquad S_{j,k}= \sum_{i=1}^n x'_{i,j}x'_{i+n,k}-x'_{i+n,j}x'_{i,k},$$ $$E_{j,k} = \sum_{i=1}^{2n} x'_{i,j} \frac{\partial}{\partial x'_{i,k}},$$ and define $\tau':\gs\go_{2m}\ra \cD(V)$ by \begin{equation}\label{explicitsp} e_{j,k+m} - e_{k,j+m} \mapsto S_{j,k}, \qquad e_{j+m,k} - e_{k+m,j} \mapsto D_{j,k}, \qquad e_{j,k} - e_{m+k, m+j} \mapsto E_{j,k} + n\delta_{j,k},\end{equation} which extends to an algebra homomorphism $\tau': U(\gs\go_{2m}) \ra \cD(V)$. This induces an action $\gs\go_{2m}\ra \text{End}(\cD(V))$ by derivations of degree zero, where $\eta\in \gs\go_{2m}$ acts by $[\tau'(\eta),-]$. As an $\gs\go_{2m}$-module, $\text{gr}(\cD(V))\cong \text{Sym} \big(\bigoplus_{i=1}^{2n} U_i \big)$, where $U_i$ is a copy of the standard $\gs\go_{2m}$-module $\mathbb{C}^{2m}$ with basis $\{x'_{i,1},\dots, x'_{i,m}, \frac{\partial}{\partial x'_{i,1}},\dots, \frac{\partial}{\partial x'_{i,m}}\}$. If $m=1$, $\cD(V)^{\text{Sp}_{2n}} = \mathbb{C}[e]$ where $e = \sum_{i=1}^{2n} x'_i \frac{\partial}{\partial x'_i}$, and if $m\geq 2$, $\cD(V)^{\text{Sp}_{2n}} = \tau'(U(\gs\go_{2m}))$. As for the double commutant $\text{Com}(\cD(V)^{\text{Sp}_{2n}},\cD(V))$, we have the following three cases, by $\text{Sp}_{2n} - \gs\go_{2m}$ Howe duality:

\begin{itemize}
\item If $m=1$, $\text{Com}(\cD(V)^{\text{Sp}_{2n}}, \cD(V)) = \tilde{\tau}(U(\gg\gl_{2n}))$ where $\tilde{\tau}$ is the extension of $\tau$ to $U(\gg\gl_{2n})$.

\item If $2\leq m\leq n$, $\text{Com}(\cD(V)^{\text{Sp}_{2n}}, \cD(V)) = \cD(V)^{\gs\go_{2m}}$, which is generated by $\tau(U(\gs\gp_{2n}))$ together with all $2m \times 2m$ determinants $d_I$, where $I$ corresponds to a choice of $2m$ distinct modules from the collection $\{U_i | i=1,\dots, 2n\}$. 

\item If $m>n$, $\text{Com}(\cD(V)^{\text{Sp}_{2n}}, \cD(V)) = \cD(V)^{\gs\go_{2m}} = \tau(U(\gs\gp_{2n}))$, so that $\cD(V)^{\text{Sp}_{2n}}$ and $\tau(U(\gs\gp_{2n}))$ form a pair of mutual commutants inside $\cD(V)$.
\end{itemize}

The vertex algebra analogues of $\tau$ and $\tau'$ are 
$$\hat{\tau}: V_{-m}(\gs\gp_{2n}) \ra \cS(V), \qquad \hat{\tau}': V_{-n}(\gs\go_{2m}) \ra \cS(V).$$ Here $\hat{\tau}$ is given by (\ref{deftheta}), and $\hat{\tau}'$ is given by
$$X^{e_{j,k+m} - e_{k,j+m}}\mapsto \sum_{i=1}^n :\gamma^{x'_{i,j}}\gamma^{x'_{i+n,k}}:-:\gamma^{x'_{i+n,j}}\gamma^{x'_{i,k}}:,$$ $$   X^{e_{j+m,k} - e_{k+m,j}}\mapsto \sum_{i=1}^n :\beta^{x_{i,j}}\beta^{x_{i+n,k}}:-:\beta^{x_{i+n,j}}\beta^{x_{i,k}}:,$$ $$X^{e_{j,k} - e_{m+k, m+j}} \mapsto \sum_{i=1}^{2n} :\gamma^{x'_{i,j}} \beta^{x'_{i,k}}:.$$ As usual, let $\Theta = \hat{\tau}(V_{-m}(\gs\gp_{2n}))$ and $\Theta' = \hat{\tau}'(V_{-n}(\gs\go_{2m}))$. 

\begin{thm} For $m=1$, $\cS(V)^{\gs\gp_{2n}[t]} = \cH$, where $\cH$ is the copy of the Heisenberg algebra generated by $e= \sum_{i=1}^{2n} : \beta^{x_i}\gamma^{x'_i}:$. For $2\leq m \leq n+1$, $\cS(V)^{\gs\gp_{2n}[t]}  = \Theta'$. Finally, for $m>n$, $\text{Com}(\cS(V)^{\gs\gp_{2n}[t]} ,\cS(V)) = \cS(V)^{\gs\go_{2m}[t]} = \Theta$, so $\Theta$ is a member of a Howe pair. \end{thm}

\begin{proof} The first statement and the second statement in the case $m\leq n$ are clear because $V\oplus V^*$ is coregular in these cases and \eqref{injgamm} is easily seen to be surjective. The case $m=n+1$ in which $(V\oplus V^*)/\!\!/ \text{Sp}_{2n}$ is a hypersurface follows from Theorem 4.5 of \cite{LSS}. The proof of the third statement is the same as the proof of Theorem \ref{howeson}. Note that for $m\leq n$, $\Theta' \cong V_{-n}(\gs\go_{2m})$, and for $m = n+1$, all relations among the generators of $\Theta'$ and their derivatives are consequences of a single relation.
\end{proof}

\section{Commutants inside $bc$-systems and $bc\beta\gamma$-systems}

In some cases, our methods can be applied to describe commutant subalgebras of $bc$-systems and $bc\beta\gamma$-systems. Given a finite-dimensional vector space $V$, the $bc$-system $\cE(V)$ was introduced in \cite{FMS}. It is the unique vertex superalgebra with odd generators $b^{x}$, $c^{x'}$ for $x\in V$, $x'\in V^*$, which satisfy
\begin{equation} \begin{split} &b^x(z)c^{x'}(w)\sim\langle x',x\rangle (z-w)^{-1}, \qquad c^{x'}(z)b^x(w)\sim \langle x',x\rangle (z-w)^{-1},\\ &b^x(z)b^y(w)\sim 0, \qquad c^{x'}(z)c^{y'}(w)\sim 0.\end{split} \end{equation} We give $\cE(V)$ the conformal structure $L_{\cE} = -\sum_{i=1}^n :b^{x_i}\partial c^{x'_i}:$ under which $b^{x_i}$, $c^{x'_i}$ are primary of conformal weights $1$, $0$, respectively. As usual, $\{x_1,\dots, x_n\}$ is a basis for $V$ and $\{x'_1,\dots,x'_n\}$ is the dual basis for $V^*$.

Let $G$ be a connected, reductive complex group with Lie algebra $\gg$, and let $V$ be a $G$-module. There is a vertex algebra homomorphism $V_{1}(\gg,B)\rightarrow \cE(V)$ which is analogous to (\ref{deftheta}), given by
\begin{equation} \label{defthetabc} X^{\xi} \mapsto \theta_{\cE}^{\xi} =  \sum_{i=1}^n :c^{x'_i} b^{\rho(\xi)(x_i)}:.\end{equation} Here $B(\xi,\eta) = \text{Tr}(\rho(\xi) \rho(\eta))$ and $\rho:\gg\rightarrow \text{End}(V)$ is the induced representation of $\gg$. There is a good increasing filtration on $\cE(V)$, where $\cE(V)_{(r)}$ is spanned by normally ordered monomials in $b^{x_i}$, $c^{x'_i}$ and their derivatives, of length at most $r$. We have $\cE(V) \cong \text{gr}(\cE(V))$ as linear spaces, and as supercommutative algebras we have $$\text{gr}(\cE(V)) \cong \bigwedge \bigoplus_{k\geq 0} (V_k \oplus V^*_k), \qquad V_k = \{b^x_k|\ x\in V\}, \qquad V^*_k = \{c^{x'}_k|\ x'\in V^*\}.$$ Here $b^x_k$, $c^{x'}_k$ are the images of $\partial^k b^x$, $\partial^k c^{x'}$ in $\text{gr}(\cE(V))$. The map $V_1(\gg,B)\rightarrow \cE(V)$ induces an action of $\gg[t]$ on $\text{gr}(\cE(V))$ by superderivations of degree zero, defined by 
\begin{equation}\label{actionthetabc}\xi t^r(b^x_i) = \lambda^r_i b^{\rho(\xi)(x)}_{i-r}, \qquad \xi t^r(c^{x'}_i) = \lambda^r_i c^{\rho^*(\xi)(x')}_{i-r}, \qquad \lambda^r_i = \bigg\{ \begin{matrix} \frac{i!}{(i-r)!}  & 0\leq r\leq i \cr 0 & r>i \end{matrix}.\end{equation}

We may consider the commutant $\text{Com}(\Theta, \cE(V)) = \cE(V)^{\gg[t]}$, where $\Theta$ is the image of \eqref{defthetabc}. More generally, given representations $V$ and $W$ of $G$, the $bc\beta\gamma$-system $\cE(V)\otimes \cS(W)$ carries an action of $V_1(\gg,B)$ given by \begin{equation} \label{thetabcbg} X^{\xi} \mapsto \theta^{\xi}_{\cE\otimes \cS} = \theta^{\xi}_{\cE} \otimes 1 + 1\otimes \theta^{\xi}_{\cS}.\end{equation} Here $\theta^{\xi}_{\cS}$ is given by \eqref{deftheta}, and $B(\xi,\eta) = \text{Tr}(\rho_1(\xi) \rho_1(\eta) ) -\text{Tr}(\rho_2(\xi) \rho_2(\eta))$, where $\rho_1: \gg\rightarrow \text{End}(V)$ and $\rho_2: \gg\rightarrow \text{End}(W)$ are the induced representations of $\gg$. We may also consider $\text{Com}(\Theta,\cE(V)\otimes \cS(W)) = (\cE(V)\otimes \cS(W))^{\gg[t]}$, where $\Theta$ is the image of \eqref{thetabcbg}.

In order to study these commutants, we need to extend our result on invariant theory and arc spaces to the case of odd as well as even variables. Given a $G$-representation $V$, let $V^*_j\cong V^*$ for $j\geq 0$, and fix a basis $\{x_{1,j},\dots, x_{n,j}\}$ for $V^*_j$. Let $S = \text{Sym}(\bigoplus_{j\geq 0} V^*_j)$. The map $\cO(V_{\infty}) \rightarrow S$ sending $x_i^{(j)}\mapsto x_{i,j}$ is an isomorphism of differential algebras, where the differential $D$ on $S$ is given by $D(x_{i,j}) = x_{i,j+1}$.

For $j\geq 0$, let $\tilde{V}^*_j\cong V^*$ and let $L= \bigwedge \bigoplus_{j\geq 0} \tilde{V}^*_j$. Fix a basis $\{y_{1,j},\dots, y_{n,j}\}$ for $\tilde{V}^*_j$ and extend the differential on $S$ to a super differential $D$ on $S\otimes L$, defined on generators by $D(y_{i,j}) = y_{i,j+1}$. There is an action of $G_{\infty}$ on $S\otimes L$, and we may consider the invariant ring $(S\otimes L)^{G_{\infty}}$. Let $S_0 = \text{Sym}(V^*_0)\subset S$ and $L_0 = \bigwedge (\tilde{V}^*_0)\subset L$, and let $\bra (S_0\otimes L_0)^G\ket$ be the differential algebra generated by $(S_0\otimes L_0)^G$, which lies in $(S\otimes L)^{G_{\infty}}$.

Since $G$ acts on the direct sum $V^{\oplus k}$ of $k$ copies of $V$, we have a map \begin{equation} \label{kcopiesofv}\cO(\quot{V^{\oplus k}} G)_{\infty} \rightarrow \cO(V^{\oplus k}_{\infty})^{G_{\infty}}.\end{equation} 
\begin{thm} \label{oddgen} Suppose that \eqref{kcopiesofv} is an isomorphism for all $k\geq 1$. Then $(S\otimes L)^{G_{\infty}}=\bra (S_0\otimes L_0)^G\ket$. \end{thm}

\begin{proof} Define a gradation on $S\otimes L$ by $\deg (x_{i,j}) = 0$ and $\deg (y_{i,j}) = 1$. For $p \in (S\otimes L)^{G_{\infty}}$, we can assume $p$ is homogeneous of degree $d$. Then $p$ can be written in the form $$p=\sum f y_{i_1,j_1}\wedge \cdots \wedge y_{i_d,j_d},$$ where $y_{i_t,j_t} \in \tilde{V}^*_{j_t}$ and $f\in S$.

For $a=1,\dots,d$, let $S^a$ be the copy of $S$ with generators $z^a_{i,j}$ for $i=1,\dots, n$, and $j\geq 0$. Define $q\in S \otimes (S^1\otimes \cdots \otimes S^d)$ by
$$q=\sum f \sum_{\sigma \in  \mathfrak{S}_d} \text{sgn} (\sigma) z^1_{i_{\sigma(1)}, j_{\sigma(1)}} \otimes \cdots \otimes z^d_{i_{\sigma(d)}, j_{\sigma(d)}},$$
where $\mathfrak{S}_d$ is the permutation group of $d$ elements. Since $p\in (S\otimes L)^{G_{\infty}}$, it follows that $q\in (S\otimes (S^1\otimes \cdots \otimes S^d))^{G_{\infty}}$. We have $$(S\otimes (S^1\otimes \cdots \otimes S^d))^{G_{\infty}} \cong \cO(W_{\infty})^{G_{\infty}} \cong \cO((\quot WG)_{\infty}),$$ where $W = \bigoplus_{i=1}^{d+1} V_i$ and $V_i \cong V$. Since $\cO(\quot WG) \cong (S_0 \otimes (S^1_0\otimes \cdots \otimes S^d_0))^G$, where $S^a_0 \cong S_0$, $q$ lies in the differential algebra $\bra (S_0 \otimes (S^1_0\otimes \cdots \otimes S^d_0))^G\ket$ generated by $(S_0 \otimes (S^1_0\otimes \cdots \otimes S^d_0))^G$. It follows that $p\in \bra(S_0\otimes L_0)^G\ket$. \end{proof}

The following result is an immediate consequence of Theorem \ref{oddgen}.

\begin{cor} \label{odd} Let $G = \text{SL}_2$, and let $V$ be the sum of finitely many copies of the standard representation. Then $(L\otimes S)^{G_{\infty}} =  \bra(S_0\otimes L_0)^G\ket$, where $L$, $S$, $L_0$, and $S_0$ are as above.
\end{cor}

The analogue of Theorem \ref{minrelations} holds in the setting of $bc$-systems and $bc\beta\gamma$-systems, and the proof is the same. 

For $G = \text{SL}_2$ and $V$ the sum of $m$ copies of $\mathbb{C}^2$ with basis $\{x_{1,j}, x_{2,j}|\ j=1,\dots, m\}$ there is an induced vertex algebra homomorphism $V_m(\gs\gl_2)\rightarrow \cE(V)$ given in the standard basis $x,y,h$ for $\gs\gl_2$ by $$X^{x} \mapsto \sum_{j=1}^m :b^{x_{1,j}} c^{x'_{2,j}}: , \qquad X^{y} \mapsto \sum_{j=1}^m :b^{x_{2,j}} c^{x'_{1,j}}:,\qquad X^{h}\mapsto  \sum_{j=1}^m :b^{x_{1,j}} c^{x'_{1,j}}: - :b^{x_{2,j}} c^{x'_{2,j}}:.$$
Let $\Theta$ be the image of $V_m(\gs\gl_2)$ inside $\cE(V)$.

\begin{thm} \label{oddsln} For all $m\geq 1$, $\text{Com}(\Theta,\cE(V)) = \cE(V)^{\gs\gl_2[t]}$ has a minimal strong generating set \begin{equation} \label{oddslngenerators} \begin{split} &\psi^{i,j} = \ :b^{x_{1,i}} c^{x'_{1,j}}: + :b^{x_{2,i}} c^{x'_{2,j}}:,\qquad 1\leq i \leq m,\qquad 1 \leq j\leq m,\\ & D_{k,l} = \ :b^{x_{1,k}} b^{x_{2,l}}:+ : b^{x_{1,l}} b^{x_{2,k}}:,\quad D'_{k,l} =\ :c^{x'_{1,k}} c^{x'_{2,l}}:+ : c^{x'_{1,l}} c^{x'_{2,k}}:,\quad 1 \leq k\leq l \leq m.\end{split} \end{equation}
The vertex algebra generated by the $\psi^{i,j}$ is the image of the map $V_2(\gg\gl_m)\rightarrow \cE(V)$ induced by the right action of $\text{GL}_m$ on $V$. All normally ordered polynomial relations among the generators \eqref{oddslngenerators} and their derivatives come from Weyl's second fundamental theorem of invariant theory for the standard representation of $\text{SL}_2$, in the sense of Lemma \ref{idealrecon}.\end{thm}

\begin{proof} We have an injective homomorphism of differential supercommutative rings 
\begin{equation} \label{sing} \text{gr}(\cE(V)^{\gs\gl_2[t]}) \rightarrow \text{gr}(\cE(V))^{\gs\gl_2[t]} \cong \bigwedge \bigoplus_{k\geq 0} (V_k \oplus V^*_k)^{\gs\gl_2[t]},\end{equation} By Corollary \ref{odd}, $\bigwedge \bigoplus_{k\geq 0} (V_k \oplus V^*_k)^{\gs\gl_2[t]}$ is generated as a differential algebra by its weight-zero component $\bigwedge(V_0\oplus V^*_0)^{\text{SL}_2}$, which is generated by $q_{i,j} = x_{1,i} x'_{1,j} + x_{2,i} x'_{2,j}$ for $i,j = 1,\dots,m$, together with $d_{k,l} = x_{1,k} x_{2,l} + x_{1,l}x_{2,k}$ and $d'_{k,l} = x'_{1,k} x'_{2,l} + x'_{1,l}x'_{2,k}$ for $1 \leq k\leq l\leq m$. The corresponding elements of $\cE(V)$ are precisely $\psi^{i,j}, D_{k,l}, D'_{k,l}$, and a calculation shows that these elements lie in $\cE(V)^{\gs\gl_2[t]}$. Therefore (\ref{sing}) is an isomorphism, and the claim follows. \end{proof}

Next, we consider the $bc\beta\gamma$-system $\cE(V)\otimes \cS(W)$, where $V$ is the sum of $r$ copies of $\mathbb{C}^2$ with basis $\{x_{1,j},x_{2,j}|\ j=1,\dots, r\}$, and $W$ is the sum of $s$ copies of $\mathbb{C}^2$ with basis $\{y_{1,k}, y_{2,k}|\ k=1,\dots, s\}$. The actions of $V_r(\gs\gl_2)$ and $V_{-s}(\gs\gl_2)$ on $\cE(V)$ and $\cS(W)$, respectively, give rise to an action of $V_{r-s}(\gs\gl_2)$ on $\cE(V)\otimes \cS(W)$, given by $$X^{x} \mapsto \sum_{j=1}^r :b^{x_{1,j}} c^{x'_{2,j}}: - \sum_{k=1}^s :\beta^{y_{1,k}} \gamma^{y_{2,k}}: , \qquad X^{y} \mapsto \sum_{j=1}^r :b^{x_{2,j}} c^{x'_{1,j}}: - \sum_{k=1}^s :\beta^{y_{2,k}} \gamma^{y_{1,k}}:,$$
 $$X^{h}\mapsto  \sum_{j=1}^r :b^{x_{1,j}} c^{x'_{1,j}}: - :b^{x_{2,j}} c^{x'_{2,j}}: - (\sum_{k=1}^s :\beta^{y_{1,k}} \gamma^{y'_{1,k}}: - :\beta^{y_{2,k}} \gamma^{y'_{2,k}}:).$$ Let $\Theta$ be the image of $V_{r-s}(\gs\gl_2)$ inside $\cE(V)\otimes \cS(W)$.

\begin{thm} \label{evenoddsln} For all $r,s \geq 1$, $\text{Com}(\Theta,\cE(V)\otimes \cS(W)) = (\cE(V)\otimes \cS(W))^{\gs\gl_2[t]}$ has a minimal strong generating set \begin{equation} \label{evenoddslnequation} \begin{split} & \psi_{\beta\gamma}^{i,j} = \ :\beta^{y_{1,i}} \gamma^{y'_{1,j}}: + :\beta^{y_{2,i}} \gamma^{y'_{2,j}}:, \qquad 1\leq i \leq s, \qquad 1 \leq j\leq s,\\
& \psi_{bc}^{k,l} =\ :b^{x_{1,k}} c^{x'_{1,l}}: +  :b^{x_{2,k}} c^{x'_{2,l}}:, \qquad 1\leq k \leq r, \qquad 1 \leq l\leq r,\\
& \psi_{\beta c}^{i,k} = \ :\beta^{y_{1,i}} c^{x'_{1,k}}: + :\beta^{y_{2,i}} c^{x'_{2,k}}:, \qquad 1\leq i \leq s, \qquad 1 \leq k\leq r,\\
&\psi_{b\gamma}^{k,i} = \ :b^{x_{1,k}} \gamma^{y'_{1,i}}: + :b^{x_{2,k}} \gamma^{y'_{2,i}}:, \qquad 1\leq k \leq r, \qquad 1\leq i\leq s,\\
& D_{k,l} =\ :b^{x_{1,k}} b^{x_{2,l}}: + :b^{x_{1,l}} b^{x_{2,k}}:,\quad D'_{k,l} =\ :c^{x'_{1,k}}c^{x'_{2,l}}: + :c^{x'_{1,l}}c^{x'_{2,k}}:, \quad 1\leq k\leq l \leq r,\\
&E_{i,k} =\ :\beta^{x_{1,i}} b^{x_{2,k}} : - :\beta^{x_{2,i}} b^{x_{1,k}}:, \quad E'_{i,k} =\ :\gamma^{x'_{1,i}}c^{x'_{2,k}}: - :\gamma^{x'_{2,i }}c^{x'_{1,k }}:, \quad 1\leq i\leq s,\quad 1\leq k \leq r,\\
& F_{i,j} =\ :\beta^{x_{1,i}} \beta^{x_{2,j}}: - :\beta^{x_{1,j}} \beta^{x_{2,i}}:, \quad  F'_{i,j}=\ :\gamma^{x'_{1,i}}\gamma^{x'_{2,j}}: - :\gamma^{x'_{1,j}}\gamma^{x'_{2,i}}: , \quad 1\leq i<j \leq s.\end{split} \end{equation}
The elements $\psi_{\beta\gamma}^{i,j}$, $\psi_{bc}^{k,l}$, $\psi_{b\gamma}^{k,i}$, $\psi_{\beta c}^{i,k}$, generate an affine vertex superalgebra associated to $\gg\gl(r|s)$. 
All normally ordered polynomial relations among the generators \eqref{evenoddslnequation} and their derivatives come from Weyl's second fundamental theorem of invariant theory for the standard representation of $\text{SL}_2$, in the sense of Lemma \ref{idealrecon}.
\end{thm}

\begin{proof} An OPE calculation shows that the vertex operators \eqref{evenoddslnequation} all lie in $(\cE(V)\otimes \cS(W))^{\gs\gl_2[t]}$. By Corollary \ref{odd}, the map $$\text{gr}((\cE(V)\otimes \cS(W))^{\gs\gl_2[t]}) \hookrightarrow \text{gr}(\cE(V)\otimes \cS(W))^{\gs\gl_2[t]}$$ is an isomorphism. This proves the claim. \end{proof}

Unfortunately, we are unable to prove the analogues of Theorems \ref{oddsln} and \ref{evenoddsln} for $\SL_n$ with $n\geq 3$ because it is not known if \eqref{arcmap} is surjective. If it were surjective, this would imply that $\cE(V)^{\gs\gl_n[t]}$ and $(\cE(V)\otimes \cS(W))^{\gs\gl_n[t]}$ have similar minimal strong generating sets, with $D,D',E,E',F,F'$ replaced by elements of degree $n$ that correspond to $n\times n$ determinants. However, unlike the case $n=2$, \eqref{arcmap} is not injective in general; see Example 6.6 of \cite{LSS}. This implies that there exist normally ordered polynomial relations among the generators and their derivatives that do not come from classical relations. Similar questions remain open for representations of the other classical groups which are sums of copies of the standard representation.

\end{document}